\documentclass{amsart}
\usepackage[utf8]{inputenc}
\usepackage{amsmath, amssymb}
\usepackage{amsthm}
\usepackage[all]{xy}
\usepackage{color}
\definecolor{darkblue}{rgb}{0,0,0.6}
\usepackage[ocgcolorlinks,colorlinks=true, citecolor=darkblue, filecolor=darkblue,
linkcolor=darkblue, urlcolor=darkblue]{hyperref}
\usepackage[nameinlink,capitalise,noabbrev]{cleveref}
\usepackage{bbm}
\usepackage{colonequals}  
\usepackage{enumerate}

\makeatletter
\@namedef{subjclassname@2010}{%
  \textup{2010} Mathematics Subject Classification}
\makeatother
\title{Regular Finite Decomposition Complexity}
\author[Kasprowski]{Daniel Kasprowski}
\address{Rheinische Friedrich-Wilhelms-Universit\"{a}t Bonn, Mathematisches Institut, Endenicher Allee 60, 53115 Bonn, Germany}
\email{kasprowski@uni-bonn.de}
\thanks{The first author was partially supported by the Max Planck Society.}

 \author[Nicas]{Andrew Nicas}
\address{Department of Mathematics and Statistics, McMaster University, Hamilton, Ontario, Canada L8S 4K1}
\email{nicas@mcmaster.ca}
\thanks{The second author was partially supported by a grant from
the Natural Sciences and Engineering Research Council of Canada}
 
\author[Rosenthal]{David Rosenthal}
\address{Department of Mathematics and Computer Science, St.\ John's University, 8000 Utopia
Pkwy, Queens, NY 11439, USA}
\thanks{The third author was partially supported by a grant from the Simons Foundation, \#229577.}
\email{rosenthd@stjohns.edu}


\newcounter{commentcounter}

\date{\today}

\newcommand{\bbR}{\mathbbm{R}}

\newcommand{\bbZ}{\mathbbm{Z}}
\newcommand{\bbN}{\mathbbm{N}}
\newcommand{\bbL}{\mathbb{L}}
\newcommand{\bbK}{\mathbb{K}}

\newcommand{\mcX}{\mathcal{X}}
\newcommand{\mcY}{\mathcal{Y}}
\newcommand{\mcZ}{\mathcal{Z}}
\newcommand{\mcT}{\mathcal{T}}

\DeclareMathOperator{\id}{id}
\DeclareMathOperator{\diam}{diam}

\DeclareMathOperator{\asdim}{asdim}
\DeclareMathOperator{\mesh}{mesh}

\newcommand{\mcV}{\mathcal{V}}
\newcommand{\mcW}{\mathcal{W}}

\newcommand{\mcF}{\mathcal{F}}

\newcommand{\mcA}{\mathcal{A}}
\newcommand{\mcU}{\mathcal{U}}
\newcommand{\mcB}{\mathcal{B}}
\newcommand{\mcP}{\mathcal{P}}

\newcommand{\comment}[1]{}

\newcommand{\fB}{\mathfrak{B}}
\newcommand{\fC}{\mathfrak{C}}
\newcommand{\fD}{\mathfrak{D}}
\newcommand{\wD}{w\mathfrak{D}}

\newcommand{\fR}{\mathfrak{R}}

\numberwithin{equation}{section}
\newtheorem{thm}[equation]{Theorem}
\newtheorem{theorem}[equation]{Theorem}

\newtheorem{prop}[equation]{Proposition}
\newtheorem{cor}[equation]{Corollary}
\newtheorem{lemma}[equation]{Lemma}

\newtheorem{cor+}{Corollary}
\newtheorem{prop+}[cor+]{Proposition}
\theoremstyle{definition}

\newtheorem{example}[equation]{Example}
\newtheorem{defi}[equation]{Definition}

\newtheorem*{thmm}{Theorem}

\newtheorem{rem}[equation]{Remark}
\newtheorem{remark}[equation]{Remark}

 \newtheoremstyle{TheoremNum}
        {}{}              
        {\itshape}                      
        {}                              
        {\bfseries}                     
        {.}                             
        { }                             
        {\normalfont\bfseries\thmname{#1}\thmnote{#3}}
\theoremstyle{TheoremNum}

\begin{document}

\begin{abstract}
We introduce the notion of {\it regular finite decomposition complexity} of a metric family. This generalizes Gromov's finite asymptotic dimension and is motivated by the concept of finite decomposition complexity (FDC) due to Guentner, Tessera and Yu. Regular finite decomposition complexity implies FDC and has all the permanence properties that are known for FDC, as well as a new one called Finite Quotient Permanence. We show that for a collection containing all metric families with finite asymptotic dimension all other permanence properties follow from Fibering Permanence.
\end{abstract}

\maketitle

\baselineskip 15pt

\section{Introduction}
Guentner, Tessera and Yu introduced finite decomposition complexity (FDC) in \cite{rigidity} as a 
generalization of Gromov's finite asymptotic dimension. In this paper, we introduce the notion of {\it regular finite decomposition complexity} (abbreviated as regular FDC). As with FDC, regular FDC is a coarse geometric property of \emph{metric families}. 

A metric family is a set of metric spaces. A metric family $\mcX$ \emph{regularly decomposes} over a collection of metric families $\fC$ if there exists a family $\mcY$ with finite asymptotic dimension and a coarse map $F\colon\mcX\to\mcY$ such that for every (uniformly) bounded subfamily $\mcB$ 
of $\mcY$ the inverse image $F^{-1}(\mcB)$ lies in $\fC$.
We show in \cref{prop:ordinals} that there exists a smallest collection of metric families that is closed under regular decomposition and contains all bounded metric families. We call this collection $\fR$ and say that a metric family in $\fR$ has {\it regular FDC}. Clearly $\fR$ contains all metric families with finite asymptotic dimension. We show in \cref{lem:FDC} that regular FDC implies FDC. 

The collection $\fR$ has many {\it permanence properties}. Informally, a permanence property of a collection $\fC$ of metric families is an operation that when applied to members of $\fC$ yields another member of $\fC$. All of the permanence properties proved for FDC in \cite{fdc} also hold for regular FDC. That is, regular FDC satisfies Coarse Permanence (\cref{thm:coarseinv}), Fibering Permanence (\cref{thm:fibering}), Finite Amalgamation Permanence, Finite Union Permanence, Union Permanence and Limit Permanence (\cref{cor:rFDC_permanence}). Regular FDC also satisfies \emph{Finite Quotient Permanence} (\cref{thm:quotients1}), which FDC is not known to satisfy. 

Fibering Permanence is a particularly important permanence property. A typical special case of  Fibering Permanence is the following.
Assume  $\fC$ satisfies Fibering Permanence and $f\colon X\to Y$ is a coarse map with  $\{Y\}\in\fC$.
If for all $r>0$ the family $\{f^{-1}(B_r(y))\mid y\in Y\}$ is in $\fC$, then $\{X\}\in \fC$.
The general formulation of Fibering Permanence requires the use of general metric families and so \cref{def:fibering} is stated accordingly. We show that if a collection of metric families satisfies Fibering Permanence and contains all metric families with finite asymptotic dimension, then all of the other permanence properties mentioned above, except for Finite Quotient Permanence, are automatically satisfied. That is, we prove the following general theorem (see Theorems~\ref{thm:coarseinv-general}, \ref{thm:amalgamation-general}, \ref{thm:finunion2}, \ref{thm:union} and \ref{thm:limit_perm}).

\begin{thm}
Let $\fC$ be a collection of metric families that satisfies Fibering Permanence and contains all metric families with finite asymptotic dimension. Then $\fC$ satisfies Coarse Permanence, Finite Amalgamation Permanence, Finite Union Permanence, Union Permanence and Limit Permanence.
\end{thm}

Using this theorem and the results from \cite{guentner} and \cite{ramras-ramsey}, we obtain the following corollary which applies, in particular, to regular FDC.

\begin{cor}
Let $\mcP$ be a property of metric families that is satisfied by all metric families with finite asymptotic dimension and that is closed under fibering. Then the class of (countable) groups with $\mcP$ is closed under extensions, direct unions, free products (with amalgam) and relative hyperbolicity. Furthermore, all elementary amenable groups, all linear groups and all subgroups of virtually connected Lie groups have $\mcP$.
\end{cor}

The definition of regular decomposition is a special case of fibering, and so we immediately obtain the following theorem.

\begin{thmm}(\cref{thm:smallest_fibering})
The collection of metric families with regular FDC is the smallest collection of metric families that contains all families with finite asymptotic dimension and satisfies Fibering Permanence.
\end{thmm} 

When proving injectivity results for the assembly maps in algebraic $K$- and $L$-theory for groups with torsion, one has to work with quotients by finite groups. Hence,
Finite Quotient Permanence is a useful property to have at one's disposal. It is not hard to see that the collection of metric families with finite asymptotic dimension satisfies Finite Quotient Permanence (see \cref{prop:asdim_quotient}). Similarly, it is straightforward to verify that \emph{weak FDC} (a condition that is looser than FDC, also introduced by Guentner, Tessera and Yu), satisfies Finite Quotient Permanence as well. However, the proof techniques used to establish injectivity rely on FDC and it is not known if FDC satisfies Finite Quotient Permanence. While FDC is a special case of weak FDC, it is an open question whether or not weak FDC implies FDC. On the other hand, regular FDC implies FDC (\cref{lem:FDC}) and satisfies Finite Quotient Permanence (\cref{thm:quotients1}). Therefore, we directly obtain the following consequence of \cite[Theorems 8.1 and 9.1]{KasFDC}.

\begin{thm}
\label{thm:inj}
Let $G$ be a group with regular FDC. Assume there exists a finite dimensional model for $\underbar EG$ and a global upper bound on the order of the finite subgroups of $G$. Then the $K$-theoretic assembly map \[H_n^G(\underbar EG;\bbK_\mcA)\to K_n(\mcA[G])\]
is split injective for every additive $G$-category $\mcA$.

If $\mcA$ is an additive $G$-category with involution such that for every finite subgroup $F\leq G$ there exists $N\in\bbN$ with $K_{-N}(\mcA[F])=0$ for all $n\geq N$, then also the $L$-theoretic assembly map
\[H_n^G(\underbar EG;\bbL^{\langle-\infty\rangle}_\mcA)\to L_n^{\langle-\infty\rangle}(\mcA[G])\]
is split injective.
\end{thm}

In \cite{KasFDC} this result was obtained for groups with ``fqFDC" instead of regular FDC. A group $G$ has fqFDC if for every $n\in\bbN$ the family $\{F\backslash G\mid F\leq G, |F|\leq n\}$ has FDC. But the technical concept of fqFDC does not satisfy most of the permanence properties discussed above. Thus, one advantage of regular FDC is that is has all the permanence properties that FDC has and is strong enough to imply the above result about injectivity of the assembly maps in algebraic $K$- and $L$-theory.

In \cref{sec:prelim} we gather some preliminary facts about the coarse geometry of metric families and introduce regular finite decomposition complexity. In \cref{sec:asdim} we focus on the asymptotic dimension of metric families, and the asymptotic Assouad-Nagata dimension of metric families, extending several known facts about these notions for metric spaces to metric families. In \cref{sec:extension} we prove an Extension Theorem for metric families that plays an important role in establishing Finite Quotient Permanence for regular FDC. In \cref{sec:inheritance} we study the permanence properties of regular FDC.

The authors would like to thank the referee for several useful comments.


\section{Preliminaries}\label{sec:prelim}

We begin by recalling some elementary concepts from coarse geometry utilizing the language of \emph{metric families}. Guentner, Tessera and Yu introduced metric families in~\cite{rigidity} to define their notion of \emph{finite decomposition complexity} (\cref{def:FDC} below), a generalization of Gromov's finite asymptotic dimension. 

A \emph{metric family} is a set of metric spaces. A \emph{map of metric families}, $F\colon \mcX\rightarrow \mcY$, is a collection of functions $f\colon X\to Y$, where $X\in \mcX$ and $Y\in\mcY$, such that each element in $\mcX$ is the domain of at least one function in $F$. 

The composition $G\circ F\colon \mcX\to \mcZ$ of $G\colon \mcY\to \mcZ$ and $F\colon \mcX\to \mcY$ is the collection
$\{g\circ f\mid f\in F, g\in G, \text{and the domain of $g$ is the range of $f$}\}$.

\begin{defi}\label{coarse}
	Let $F\colon \mcX\rightarrow \mcY$ be a map of metric families.

	\begin{enumerate}[(i)]
		\item $F$ is \emph{coarse} (or \emph{uniformly expansive}) if there exists a non-decreasing function \[\rho\colon [0,\infty)\rightarrow[0,\infty)\] such that for every $X\in \mcX, x,y \in X$,  and $f\colon X\to Y$ in $F$, 
			\begin{equation*}
				d_Y(f(x), f(y))\leq \rho(d_X(x,y)).
			\end{equation*}
We call $\rho$ the \emph{control function for} $F$.

		\item $F$ is \emph{effectively proper} if there exists a proper non-decreasing function \[\delta\colon [0,\infty)\rightarrow[0,\infty)\] such that for every $X\in\mcX, x,y \in X$,  and $f\colon X\to Y$ in $F$,
			\begin{equation*}
				\delta(d_X(x, y))\leq d_Y(f(x), f(y)).
			\end{equation*}

		\item $F$ is a \emph{coarse embedding} if it is both coarse and effectively proper.

		\item $F$ is \emph{coarsely onto} if every $Y\in\mcY$ is the range of some $f\in F$ and if  there exists a $C \geq 0$  such that for every $f \colon X \to Y$ in $F$
and for every $y \in Y$ there exists an $x \in X$ such that $d_Y(f(x), y) \leq C$.

		\item $F$ is \emph{close} to $F' \colon \mcX\rightarrow \mcY$ if there exists a $C \geq 0$ with the property that  for every $f\colon X\to Y$ in  $F$ (respectively, in $F'$)
there exists an $h \colon X\to Y$ in  $F'$  (respectively, in $F$) such
that for all $x \in X$, $d_Y(f(x), h(x))\leq C$.

		\item $F$ is a \emph{coarse equivalence} if it is coarse and there exists a coarse map $G \colon \mcY\rightarrow \mcX$
such that  $G \circ F$ is  close to the identity map of $\mcX$ and $F \circ G$ is close to the identity map of $\mcY$.
	\end{enumerate}
\end{defi}

A {\it subfamily} of a metric family $\mcY$ is a metric family $\mcA$ such that every $A \in \mcA$ is a subspace of some $Y \in \mcY$. The {\it inverse image} of $\mcA$ under the map $F\colon \mcX\rightarrow \mcY$ is the subfamily of $\mcX$ given by $F^{-1}(\mcA)=\left\{ f^{-1}(A) \; | \; A\in \mcA, f\in F \right\}$.

A metric family $\mcX$ is called \emph{bounded} if $\sup_{X\in\mcX}\diam X<\infty$. The collection of all bounded metric families is denoted by $\fB$. 

Recall that a metric space $X$ is the \emph{$r$-disjoint union} of subspaces $\{X_i~|~i\in I \}$ if $X=\bigcup_{i\in I}X_i$, and for every $x\in X_i$ and  $y\in X_j$ with $i\neq j$, $d(x,y)>r$. We denote an $r$-disjoint union by 
\[X=\bigsqcup_{r\text{-disjoint}}\{X_{i}~|~i\in I\}.\]

\begin{defi}\label{def:decompose}
	Let $\fC$ be a collection of metric families. Let $n\in\bbN$ and $r>0$. A metric family $\mcX$ is {\em $(r,n)$-decomposable} over $\fC$ if for every $X \in \mcX$ there is a decomposition $X=X_0 \cup X_1 \cup \cdots \cup X_n$ such that for each~$i$, $0\leq i \leq n$, 
	$$X_i=\bigsqcup_{r\text{-disjoint}}\{X_{ij}~|~j\in J_i\},$$ 
and the metric family $\{X_{ij}~|~X\in\mcX,~0\leq i \leq n,~j\in J_i\}$ is in $\fC$. 

	The metric family $\mcX$ is {\em $n$-decomposable} over $\fC$ if $\mcX$ is $(r,n)$-decomposable over $\fC$ for every $r>0$. 
	
	A metric family $\mcX$ is \emph{strongly decomposable} over $\fC$ if it is $1$-decomposable over $\fC$. It is \emph{weakly decomposable} over $\fC$ if it $n$-decomposable over $\fC$ for some $n\in\bbN$.
\end{defi}

Notice that if $\mcX$ consists of a single metric space $X$, then the saying that $\mcX$ is {\em $n$-decomposable} over $\fB$ is precisely the statement that $X$ has asymptotic dimension at most $n$. Thus, finite asymptotic dimension can be generalized to metric families as follows. It will be discussed further in \cref{sec:asdim}.

\begin{defi}\label{def:asdim}
	Let $n\in\bbN$. The metric family $\mcX$ has {\em asymptotic dimension at most $n$}, denoted $\asdim(\mcX)\leq n$, if $\mcX$ is $n$-decomposable over~$\fB$.\footnote{This is equivalent to Bell and Dranishnikov's definition of a collection of metric spaces having finite asymptotic dimension ``uniformly" (\cite[Section 1]{BD-trees}).}
\end{defi}

Guentner, Tessera and Yu defined finite decomposition complexity as follows.

\begin{defi}\label{def:FDC}
	Let $\fD$ be the smallest collection of metric families containing $\fB$ that is closed under strong decomposition, and let $\wD$ be the smallest collection of metric families containing $\fB$ that is closed under weak decomposition. A metric family in $\fD$ is said to have {\em finite decomposition complexity} (abbreviated to ``FDC"), and a metric family in $\wD$ is said to have {\em weak finite decomposition complexity} (abbreviated to ``weak FDC").
\end{defi}

Guentner, Tessera and Yu provided other equivalent formulations of finite decomposition complexity. One such formulation is the following. Let $\fD_0=\fB$, and for each ordinal $\alpha$ greater than 0, define $\fD_\alpha$ to be the collection of metric families that is strongly decomposable (i.e., 1-decomposable) over $\bigcup_{\beta<\alpha} \fD_\beta$.

\begin{theorem}\cite[Theorem 2.4]{rigidity}\label{rigidity}
	A metric family $\mcX$ has FDC if and only if there exists a countable ordinal $\alpha$ such that $\mcX \in \fD_\alpha$.
\end{theorem}

Inspired by the work of Guentner, Tessera, and Yu, we introduce the notion of \emph{regular finite decomposition complexity}.

\begin{defi}\label{def:rFDC}
	A metric family $\mcX$ \emph{regularly decomposes} over a collection of metric families $\fC$ if there exists a family $\mcY$ with finite asymptotic dimension and a coarse map $F\colon\mcX\to\mcY$ such that for every bounded subfamily $\mcB$
of $\mcY$, the inverse image $F^{-1}(\mcB)$ lies in $\fC$.
\end{defi}

\begin{defi}\label{def:r_class}
	Let $\fR_0=\fB$, the collection of bounded metric families. For each ordinal $\alpha$ greater than 0, let $\fR_\alpha$ be the collection of metric families that regularly decomposes over $\bigcup_{\beta<\alpha} \fR_\beta$. That is, $\mcX \in \fR_\alpha$ if there exists a metric family $\mcY$ with finite asymptotic dimension and a coarse map $F\colon\mcX\to\mcY$ such that for every bounded subfamily $\mcB$ of $\mcY$, there is an ordinal $\beta<\alpha$ such that $F^{-1}(\mcB)$ lies in $\fR_\beta$. 

	Let $\fR$ be the collection of metric families that belong to $\fR_\alpha$ for some countable ordinal $\alpha$. A metric family in $\fR$ is said to have \emph{regular finite decomposition complexity}.
\end{defi}

It is straightforward to see that $\fR_1$ is the collection of metric families with finite asymptotic dimension. 

The collection $\fR$ has many nice properties. It is a subcollection of $\fD$ (\cref{lem:FDC} below), and it possesses all of the permanence properties that $\fD$ does, plus an additional permanence property concerning quotients by finite groups (\cref{thm:quotients1}) that $\fD$ is not known to satisfy. 

The first of the permanence properties, Coarse Permanence, is proved below. The remaining permanence properties are established in \cref{sec:inheritance}.

\begin{defi}\label{def:coarse_permanence}
A collection of metric families, $\fC$, satisfies \emph{Coarse Permanence} if whenever $\mcY\in\fC$ and $F\colon\mcX\to\mcY$ is a coarse embedding, then $\mcX\in\fC$.
\end{defi}

\begin{rem}
Our definition of Coarse Permanence varies from the definition of Coarse Invariance in \cite{guentner}. To be precise, by \cite[Lemma 6.1]{guentner} a collection satisfies Coarse Permanence if and only if it satisfies Coarse Invariance and Subspace Permanence as defined in \cite{guentner}.
\end{rem}

To show that $\fR$ satisfies Coarse Permanence, we prove the stronger fact that $\fR_\alpha$ satisfies Coarse Permanence for every ordinal $\alpha$.

\begin{thm}\label{thm:coarseinv}
	The collection $\fR_\alpha$ satisfies Coarse Permanence for every ordinal~$\alpha$. In particular:
	\begin{enumerate}
		\item A subfamily of a metric family in $\fR_\alpha$ also lies in $\fR_\alpha$.
		
		\item If $\mcX$ and $\mcY$ are coarsely equivalent metric families, then $\mcX \in \fR_\alpha$ if and only if $\mcY \in \fR_\alpha$.
	\end{enumerate}
\end{thm}

\begin{proof}
	We prove the theorem by induction on $\alpha$. If $\mcY\in \fR_0$ (i.e., if $\mcY$ is bounded) and $F\colon\mcX\to\mcY$ is a coarse embedding, then $\mcX$ is also bounded. Hence, the theorem is true for $\alpha=0$. 

	Now suppose the theorem holds for all ordinals $\beta<\alpha$ and let $F\colon \mcX\to \mcY$ be a coarse embedding, with $\mcY\in \fR_\alpha$. By definition, there is a $\mcZ\in\fR_1$ and a coarse map $G\colon \mcY\to \mcZ$ such that for every bounded subfamily of $\mcZ$, the inverse image under $G$ lies in $\fR_\beta$ for some $\beta<\alpha$. Consider the coarse map $G\circ F:\mcX \to \mcZ$. To prove that $\mcX\in\fR_\alpha$, we must show that for every bounded subfamily $\mcB$ of $\mcZ$ the inverse image $(G\circ F)^{-1}(\mcB)$ lies in $\fR_\beta$ for some $\beta<\alpha$. Since $F$ is a coarse embedding, the restriction of $F$ to $(G\circ F)^{-1}(\mcB)$, denoted $F\colon (G\circ F)^{-1}(\mcB)\to G^{-1}(\mcB)$, is also a coarse embedding. Therefore, $(G\circ F)^{-1}(\mcB)\in\fR_\beta$ by the induction hypothesis, and so $\mcX\in\fR_\alpha$.
\end{proof}

Coarse Permanence enables us to prove the following result, analogous to \cref{rigidity}.

\begin{prop}\label{prop:ordinals}
	The collection $\fR$ is the smallest collection of metric families that is stable under regular decomposition and contains $\fB$. More specifically, (1) if a metric family $\mcX$ regularly decomposes over $\fR$, then $\mcX$ is in $\fR$; and (2) if $\fC$ is a collection of metric families that is stable under regular decomposition and contains $\fB$, then $\fR$ is contained in $\fC$.
\end{prop}

\begin{proof}
	For (1), assume that $\mcX$ regularly decomposes over $\fR$. Then there exists a metric family $\mcY\in\fR_1$ and a coarse map $F\colon \mcX\to \mcY$ such that for every bounded subfamily $\mcB$ of $\mcY$, the inverse image $F^{-1}(\mcB)$ lies in $\fR$; that is, $F^{-1}(\mcB)$ lies in $\fR_\alpha$ for some ordinal $\alpha$. For each $n\in\bbN$, let $\alpha_n$ be an ordinal such that $F^{-1}(\mcB_n)\in\fR_{\alpha_n}$, where $\mcB_n = \{B_n(y)\mid Y\in \mcY, y\in Y\}$ is the subfamily of all balls of radius $n$ in $\mcY$. Let $\gamma$ be an ordinal with $\alpha_n<\gamma$ for all $n\in\bbN$. For every bounded subfamily $\mcB$ of $\mcY$ there exists an $n\in\bbN$ such that $\mcB$ is a subfamily of $\mcB_n$. Thus, $F^{-1}(\mcB)\in\fR_{\alpha_n}$ by Coarse Permanence~\ref{thm:coarseinv}. This shows that in fact $\mcX$ regularly decomposes over $\fR_\gamma$, and so $\mcX\in\fR_{\gamma+1}$. Therefore, $\mcX\in \fR$.

	For (2) we have to show that for every $\alpha$, $\fR_\alpha$ is contained in $\fC$. We do this by induction. By assumption, $\fR_0$ is contained in $\fC$. Now suppose $\fR_\beta$ is contained in $\fC$ for all $\beta<\alpha$. Therefore, by definition, every family $\mcX\in\fR_\alpha$ regularly decomposes over $\fC$. Since $\fC$ is stable under regular decomposition, $\mcX$ is also in $\fC$. Thus, $\fR_\alpha$ is contained in $\fC$. 
\end{proof}

\begin{rem}
The last proposition implies that $\fR$ is the same as the collection of all metric families that belong to $\fR_\alpha$ for any ordinal $\alpha$ without assuming that $\alpha$ is countable.
\end{rem}

\begin{thm}\label{lem:FDC}
	If a metric family $\mcX$ has regular FDC, then $\mcX$ has FDC.
\end{thm}

\begin{proof}
	This follows directly from the fact that a metric family with finite asymptotic dimension is in $\fD$ (see \cite[Theorem 4.1]{fdc}) and the Fibering Theorem for FDC (\cite[Theorem 3.1.4]{fdc}). Note that \cite[Theorem 4.1]{fdc} is only stated for metric spaces, but the same proof holds for metric families as well.
\end{proof}

There is a more concrete proof of \cref{lem:FDC} that also motivates the name ``regular FDC". Suppose $\mcX$ regularly decomposes over a collection $\fC$. Then there is a coarse map $F\colon \mcX\to \mcY$, where $\asdim(\mcY)<\infty$. By \cite[Proof of Theorem 4.1]{fdc} there is a coarse embedding $\mcY\to \mcT$, where each $T\in\mcT$ is a product of $\asdim (\mcY)+1$ Gromov 0-hyperbolic spaces. A Gromov 0-hyperbolic space strongly decomposes over $\fB$ (since it has finite asymptotic dimension) and therefore we can strongly decompose $\mcT$ over $\fB$ in $\asdim(\mcY)+1$ steps. This decomposition pulls back to a decomposition of $\mcX$ over $\fC$, again in $\asdim(\mcY)+1$ steps. Given a decomposition of a Gromov 0-hyperbolic space $T$ for some $r>0$, the decomposition for $r'>>r$ can be chosen in such a way that each of the decomposition pieces for $r$ is contained in one of the pieces of $r'$. Hence, the decompositions of $\mcX$ for different values of $r$ obtained in this way are related to each other. Generally, different values for $r$ yield completely independent decompositions.


\section{Asymptotic dimension of metric families}\label{sec:asdim}

In this section we generalize some facts about finite asymptotic dimension for metric spaces to metric families. They will be used in the proof of the Extension Theorem~\ref{thm:extension} in \cref{sec:extension}, and in the proofs of the permanence properties for regular FDC in \cref{sec:inheritance}.

There are several equivalent definitions of asymptotic dimension (as can be found, for example, in~\cite[Theorem 19]{bell-dranish-asdim}). In~\cite{NR15} these were generalized to provide three alternative definitions for a metric family to be $n$-decomposable over a collection of metric families. In particular, ~\cite[Proposition 3.1]{NR15} yields the following equivalent definition for a metric family to have finite asymptotic dimension.

Recall that the {\it dimension of a covering}, $\mcU$, of a metric space $X$ is the largest integer $n$ such that every point of $X$ is contained in at most $n+1$ elements of $\mcU$. The {\it Lebesgue number} of $\mcU$, $L(\mcU)$, is at least $\lambda>0$ if for every $x \in X$ the open ball $B_\lambda(x)$ in $X$ is contained in some element of $\mcU$. The {\it mesh} of $\mcU$ is $\mesh(\mcU)=\sup\{\diam(U)~|~U \in \mcU\}$.

\begin{prop}\label{prop:Lebesgue}
	A metric family $\mcX=\{X_i\}_{i\in I}$ has finite asymptotic dimension at most $n$ if and only if for every $\lambda>0$ there exists a cover $\mcU_i$ of $X_i$, for each $i\in I$, such that:
	\begin{enumerate}
		\item the dimension of $\mcU_i$ is at most $n$ for every $i\in I$;
		
		\item the Lebesgue number $L(\mcU_i)\geq\lambda$ for every $i\in I$;
		
		\item $\bigcup_{i\in I}\mcU_i$ is a metric family in $\fB$ (that is, there exists an $R>0$ such that $\mesh(\mcU_i)\leq R$ for every $i \in I$).
	\end{enumerate}
\end{prop}

Let $X$ be a metric space and $F$ a finite group acting isometrically on $X$.  We will always consider the following metric on the quotient space $F\backslash X$:
\[d(Fx,Fx') \colonequals \min_{h\in F}d_X(x,hx')\]

\begin{prop}\label{prop:asdim_quotient}
	Let $\mcX$ be a metric family and let $F$ be a finite group that acts isometrically on every $X\in\mcX$. Let $F\backslash \mcX:=\{F\backslash X\mid X\in\mcX\}$. Then
	\[\asdim(F\backslash \mcX) \leq |F|(\asdim(\mcX)+1)-1.\]
\end{prop}
\begin{proof}
	This was proved for metric spaces in the proof of \cite[Lemma 2.2]{Bartels}. It immediately generalizes to metric families. We quickly recall the proof here. 
	
	Let $n=\asdim \mcX$. Given $\lambda>0$ there exists an $n$-dimensional cover $\mcU_X$ of $X$, for each $X\in \mcX$, and an $R>0$ such that $L(\mcU_X)\geq\lambda$ and $\mesh(\mcU_X)\leq R$ for every $X\in \mcX$. Let $q_X\colon X\to F\backslash X$ denote the quotient map. Then $q_X(\mcU_X):=\{q_X(U)\mid U\in \mcU_X\}$ is a cover of $F\backslash X$. By definition of the metric on $F\backslash X$, we have that $L(q_X(\mcU_X))\geq\lambda$ and $\mesh(q_X(\mcU_X))\leq R$. For every $y\in F\backslash X$, $q_X^{-1}(y)$ contains at most $|F|$ points. Since the dimension of $\mcU_X$ is at most $n$, it follows that the dimension of $q_X(\mcU_X)$ is at most $|F|(n+1)-1$. Therefore, by \cref{prop:Lebesgue}, the asymptotic dimension of $F\backslash \mcX$ is at most $|F|(n+1)-1$.
\end{proof}

\begin{defi}
	Let $\mcY$ be a metric family. The {\it asymptotic Assouad-Nagata dimension} of $\mcY$, denoted $\asdim_{{\rm AN}}(\mcY)$, is the smallest non-negative integer $n$ with the following property. There exist non-negative constants $b$ and $M$ such that for every $Y \in \mcY$ and every $R > 0$ there exists a cover $\mcU$ of $Y$
such that $\mcU  = \mcU_0  \,  \cup \cdots \cup \,  \mcU_n$, where each collection $\mcU_i$ is $R$-disjoint and $\mesh(\mcU)  \leq M R + b$. The function $D_{\mcY}(r) = Mr +b$  is called an {\it $n$-dimensional control function} for $\mcY$. If there does not exist an $n$ satisfying the above conditions, then we define $\asdim_{{\rm AN}}(\mcY) = \infty$.
\end{defi}

Note that $\asdim(\mcY) \leq  \asdim_{{\rm AN}}(\mcY)$ for any metric family $\mcY$.

Let $\mcX$ and $\mcZ$ be metric families and let $F \colon \mcX \to \mcZ$ be a map of families. Following \cite{BDLM}, define the \emph{asymptotic dimension} of $F$, denoted $\asdim(F)$, to be
\[ 
\asdim(F):=\{\asdim(\mcA)\mid \text{ $\mcA$ is a subfamily of $\mcX$ and } \asdim(F(\mcA))=0\}. 
\]
Analogously, the  {\it asymptotic Assouad-Nagata dimension} of $F$, denoted $\asdim_{{\rm AN}}(F)$, is defined as
\[ 
\sup\{ \asdim_{{\rm AN}}(\mcA)  ~|~  \text{$\mcA$ is a subfamily of $\mcX$} \text{ and } \asdim_{{\rm AN}}(F(\mcA))=0\}. 
\]

While \cite[Theorems 1.2, 8.2, and 2.5]{BDLM}, stated below for metric families, are proved in the case of singleton families, that is, $\mcX = \{X\}$, $\mcZ = \{Z\}$ and $F = \{ f \colon X \to Z\}$,
their extensions to general metric families is straightforward.

\begin{theorem}[{\cite[Theorem 1.2]{BDLM}}]
	If $F\colon \mcX\to \mcZ$ is a coarse map, then
\[\asdim(\mcX)\leq \asdim(F) +\asdim(\mcZ).\]
\end{theorem}

The map $F$ is {\it asymptotically Lipschitz} if there exist non-negative constants $M$ and $L$ such that for all $f \colon X \to Z$ in $F$, and all $x, x' \in X$, $d_Z(f(x),f(x')) \leq M d_X(x,x') + L$.

\begin{theorem}[{\cite[Theorem 8.2]{BDLM}}]\label{thm:BDLMtheorem}
	If $F \colon \mcX \to \mcZ$  is an asymptotically Lipschitz map of metric families, then
\[ \asdim_{{\rm AN}}({\mcX}) \leq \asdim_{{\rm AN}}({F}) + \asdim_{{\rm AN}}({\mcZ}). \]
\end{theorem}

\begin{theorem}[{\cite[Theorem 2.5]{BDLM}}]\label{cor:asdim_product}
	Let $\mcX$ and $\mcY$ be metric families and let $\mcX \times \mcY = \{X \times Y~|~X\in \mcX, Y\in \mcY\}$, where each $X \times Y$ is equipped with the $\ell^1$-metric, $d^1\big((x_1,y_1),(x_2,y_2)  \big)=d_X(x_1,x_2)+d_Y(y_1,y_2)$. Then
\[\asdim(\mcX \times \mcY) \leq \asdim(\mcX)+\asdim(\mcY)\]
and
\[\asdim_{{\rm AN}}(\mcX\times \mcY)\leq \asdim_{{\rm AN}}(\mcX)+\asdim_{{\rm AN}}(\mcY).\]
\end{theorem}

Let $(X,d_X)$ be a metric space. Recall that the \emph{Gromov product} of $x,y \in X$ with respect to a base point $p \in X$ is
\[ (x|y)_p  \colonequals \tfrac{1}{2} \left( d_X(x,p )+ d_X(y,p) - d_X(x,y) \right). \]
The  space $X$ is \emph{Gromov 0-hyperbolic} if for all $x,y,z\in X$,
\[ (x|z)_p \geq \min\{(x|y)_p, \, (y|z)_p\}. \]

\begin{lemma}\label{lem:productof0hyp}
	Let $m$ be a positive integer and $\mcT$ be a metric family such that every $T \in \mcT$ is an  $\ell^1$-metric product of at most $m$ Gromov $0$-hyperbolic spaces. Then $\asdim_{{\rm AN}}(\mcT) \leq m$.
\end{lemma}

\begin{proof}
	In the proof of \cite[Proposition 9.8]{roe} it is shown that if $X$ is a Gromov $0$-hyperbolic space then  $D_X(r)=3r$ is a $1$-dimensional control function for $X$.
Note that $D_X$ is independent of $X$.
Let $Z \in \mcT$.  Then  $Z = X_1 \times \cdots \times X_m$, where each  $X_j$ is a  Gromov $0$-hyperbolic space.
Observe that $f(n) \colonequals 3^{n-1} + 3^{n-2} -1$ is the solution to the linear recurrence $f(n) = 3 f(n-1) + 2$,  $n \geq 3$ and $f(2) = 3$.
By \cite[Theorem 2.4]{BDLM},
$D^{(m)}_{X_j}(r) = f(m+1)r$ is a {\it $(1,m+1)$-control function} for $X_j$.
By definition, this means that for any $r > 0$  there is a cover $\mcU^j = \mcU_0^j  \cup \cdots \cup \mcU_{m}^j$ of $X_j$ such that
	\begin{enumerate}
		\item each $\mcU_i^j $ is $r$-disjoint,
		
		\item $\mesh (\mcU^j )\leq D^{(m)}_{X_j}(r)$,
		
		\item each $x \in X_j$ belongs to at least $m$ elements of $\mcU^j$.
	\end{enumerate}
For $0\leq i \leq m$, let $\mcU_i = \{ U_1 \times \cdots \times U_m ~|~ U_j \in \mcU^j_i, \,  1\leq j \leq m\}$.
Then each $\mcU_i$ is $r$-disjoint and, by the third property above,
$\mcU = \mcU_0  \cup \cdots \cup \mcU_{m}$ is a cover of $Z = X_1 \times \cdots \times X_m$.
Hence, $D_Z (r) = mf(m+1)r$ is an $m$-dimensional control function for $Z$.
Note that $D_Z (r)$ is independent of $Z \in \mcT$. 
Thus, $D_{\mcT}(r) = mf(m+1)r$ is an $m$-dimensional control function for $\mcT$ and so $\asdim_{{\rm AN}}(\mcT) \leq m$.
\end{proof}

\begin{remark}
	Throughout we have used the $\ell^1$-metric for products. However, when working with finite products there is flexibility in the choice of the metric on the product. 
Given metric spaces $\big\{ (X_i, d_{X_i})\big\}^m_{i=1}$ and an extended real number $1 \leq p \leq \infty$,
their {\it $\ell^p$-metric product} is $X = X_1 \times \cdots \times X_m$
with the metric
\begin{displaymath}
d^p_X\big((x_1, \ldots, x_m), (y_1, \ldots, y_m) \big)  \colonequals  \left\{
        \begin{array}{cl}
                  \left( \displaystyle\sum^m_{i=1} d_{X_i}(x_i, y_i)^p \right)^{1/p} & \mbox{if } 1 \leq p  < \infty, \\
                  & \\
                  \max\big\{ d_{X_i}(x_i, y_i) \big\}^m_{i=1} & \mbox{if } p= \infty.
        \end{array}\right.
\end{displaymath}

Note that for $1 \leq p \leq q \leq \infty$ we have the well-known inequalities 
$d^q_X \leq d^p_X \leq m^{1/p - 1/q} d^q_X$  (where, by convention, $1/\infty =0$)
and so these metrics are bi-Lipschitz equivalent with Lipschitz constants depending only on $p$, $q$ and $m$.
\end{remark}


\section{An extension theorem}\label{sec:extension}

The goal of this section is to prove the Extension Theorem~\ref{thm:extension}. \cref{cor:extension}, which follows from the Extension Theorem, is needed to establish Finite Union Permanence (\cref{def:finunion}) of $\fR_\alpha$, for every ordinal $\alpha$ (\cref{thm:finunion2alpha}).  An important step in proving the Extension Theorem is \cref{thm:coneasdimAN}, which states that if $\mcY$ is a metric family with $\asdim_{{\rm AN}}(\mcY) < \infty$, then $\asdim_{{\rm AN}}(C(\mcY)) \leq  \asdim_{{\rm AN}}(\mcY) +1$, where $C(\mcY)$ is the \emph{cone of} $\mcY$ (\cref{def:cone_family}).

Recall that the class $\fR_1$ coincides with the collection of all metric families with finite asymptotic dimension.

\begin{thm}[Extension Theorem]\label{thm:extension}
	Let $\mcX$ be a metric family with a decomposition $X=X_0\cup X_1$ for each $X\in\mcX$, $\mcX_0=\{X_0\}_{X\in\mcX}$, $\mcY\in \fR_1$ and $F\colon \mcX_0\to \mcY$ be a coarse map. Then there is a $\mcY'\in\fR_1$, a coarse embedding $\Theta\colon \mcY\to \mcY'$ and a coarse map $F'\colon \mcX\to \mcY'$ such that the diagram
\[\xymatrix{
\mcX_0\ar[r]^-F\ar[rd]_-{F'|_{\mcX_0}}&\mcY\ar[d]^-\Theta\\
&\mcY'
} \]
coarsely commutes. That is, $\Theta \circ F$ is close to $F'|_{\mcX_0}$.
\end{thm}

\begin{cor}\label{cor:extension}
	Let $\mcX$ be a metric family with a decomposition $X=\bigcup_{i=0}^nX_i$ for each $X\in\mcX$. For each $0\leq i \leq n$, let $\mcX_i=\{X_i\}_{X\in\mcX}$ and let $F_i\colon \mcX_i\to \mcY_i$ be coarse maps with $\mcY_i\in\fR_1$. Then there is a coarse map $F\colon\mcX\to \mcY$ such that $\mcY\in\fR_1$ and for every bounded subfamily $\mcB$ of $\mcY$ there are bounded subfamilies $\mcB_i$ of $\mcY_i$ such that $F^{-1}(\mcB)$ coarsely embeds into $\{\bigcup_{i=0}^nC_i\mid C_i\in F_i^{-1}(\mcB_i)\}$.
\end{cor}

\begin{proof}
	By induction it suffices to prove the case $n=1$. Let $F'_i\colon \mcX\to \mcY'_i$ be as in \cref{thm:extension}. Define $F\colon \mcX\to \mcY:=\{Y'_0\times Y'_1\mid Y_i'\in\mcY_i'\}$ to be the product of $F_0'$ and $F_1'$. Since $\Theta_0\circ F_0$ is close to $F'_0|_{\mcX_0}$, for every bounded subfamily $\mcB'_0$ of $\mcY_0'$ the inverse image $(F_0')^{-1}(\mcB_0')$ coarsely embeds into $\{C_0\cup X_1\mid C_0\in F_0^{-1}(\Theta_0^{-1}(\mcB'_0)), X_1\in \mcX_1\}$. Because $\Theta_0$ is a coarse embedding, the subfamily $\Theta_0^{-1}(\mcB'_0)$ of $\mcY_0$ is bounded.
The analogous statements hold for bounded subfamilies $\mcB_1'$ of $\mcY_1'$.
Every bounded subfamily $\mcB$ of $\mcY$ is a subfamily of $\{B'_0\times B'_1\mid B_i'\in\mcB_i'\}$ for bounded subfamilies $\mcB_i'$ of $\mcY_i'$. Hence, $F^{-1}(\mcB)$ is a subfamily of $(F_0')^{-1}(\mcB_0')\cap (F_1')^{-1}(\mcB_1')$ which coarsely embeds into $\{(C_0\cup (X_1\setminus X_0))\cap ((X_0\setminus X_1)\cup C_1)\mid X_i\in\mcX_i, C_i\in F_i^{-1}(\mcB_i)\}=\{C_0\cup C_1\mid C_i\in F_i^{-1}(\mcB_i)\}$, where $\mcB_i$ denotes the bounded subfamilies $\Theta_i^{-1}(\mcB_i')$ of $\mcY_i$. 
\end{proof}

An important tool used to prove the Extension Theorem~\ref{thm:extension} is the notion of the \emph{cone} of a metric family $\mcY$ (\cref{def:cone_family}).

Let $(Y,d_Y)$ be a metric space and let $\rho\colon[0,\infty)\to[0,\infty)$ be a non-decreasing function. Let $C(Y)$ denote the space $Y\times[0,\infty)$ with the metric
$d_{C(Y)}$ generated by the symmetric, non-negative function
\[d'((y,t),(y',t'))=|t-t'|+\tfrac{d_Y(y,y')}{\max\{\rho(\max\{t,t'\}),1\}}\]
via the {\it chain condition}. That is,
\[
d_{C(Y)}((y,t),(y',t'))= \inf \, \left\{\sum_{i=0}^{n-1}d'((y_i,t_i),(y_{i+1},t_{i+1}))\right\},
\]
where the infimum is taken over all finite sequences (``chains'') of the form
\[ \{ (y_i, t_i) \in Y\times[0,\infty) ~|~ i=0, \ldots, n\} \] 
with $(y,t) = (y_0, t_0)$  and $(y',t') = (y_n, t_n)$.

\begin{rem}
		Note that the cone $C(Y)$ depends on the choice of $\rho$, which we suppress from the notation. For different choices of $\rho$ the resulting cones will not necessarily be coarsely equivalent. For example, for constant $\rho$ the cone is quasi-isometric to the product $Y\times[0,\infty)$, while this is in general not the case if $\rho$ is proper. Most of the following results hold for any choice of $\rho$, but for \cref{prop:extensiontocone} we will have to use a specific choice of $\rho$.
\end{rem}

\begin{defi}\label{def:cone_family}
	Given a metric family $\mcY$, define the \emph{cone of} $\mcY$, denoted $C(\mcY)$, to be the metric family $\{C(Y)\}_{Y\in\mcY}$.
\end{defi}

\begin{lemma}\label{lem:conemetric}
	For $(y,t),(y',t')\in C(Y)$ we have
\[
d_{C(Y)}((y,t),(y',t')) ~=~ \inf_{s\geq\max\{t,t'\}} \left\{2s-t-t'+\tfrac{d_Y(y,y')}{\max\{\rho(s),1\}}\right\}.
\]
\end{lemma}

\begin{proof}
	Since $(y,t), \, (y,s), \, (y', s), \, (y',t')$ is a chain from $(y,t)$ to $(y',t')$
we have that
	\begin{align*}
		d_{C(Y)}((y,t),(y',t')) & \leq d'((y,t),(y,s))+d'((y,s),(y',s))+d'((y',s), (y',t'))\\
		&=2s-t-t'+\tfrac{d_Y(y,y')}{\max\{\rho(s),1\}},
	\end{align*}
provided $s\geq \max\{t,t'\}$. It follows that
\[
d_{C(Y)}((y,t),(y',t')) ~\leq~ \inf_{s\geq\max\{t,t'\}} \left\{2s-t-t'+\tfrac{d_Y(y,y')}{\max\{\rho(s),1\}}\right\}.
\]
Let $\epsilon >0$.  Then there exists a chain $\{ (y_i,t_i) ~|~ i=0,\ldots,n\}$ with $(y,t) = (y_0, t_0)$  and $(y',t') = (y_n, t_n)$ such that
\[\sum_{i=0}^{n-1}d'((y_i,t_i),(y_{i+1},t_{i+1}))<d_{C(Y)}((y,t),(y',t'))+\epsilon.\]
Let $s'=\max_i t_i$. Then
	\begin{align*}
		\sum_{i=0}^{n-1}d'((y_i,t_i),(y_{i+1},t_{i+1}))&=\sum_{i=0}^{n-1}|t_i-t_{i+1}|+\sum_{i=0}^{n-1}\tfrac{d_Y(y_i,y_{i+1})}{\max\{\rho(\max\{t_i,t_{i+1}\}),1\}}\\
		&\geq |t_0-s'|+|t_n-s'|+\sum_{i=0}^{n-1}\tfrac{d_Y(y_i,y_{i+1})}{\max\{\rho(s'),1\}}\\
		&\geq 2s'-t-t'+\tfrac{d_Y(y,y')}{\max\{\rho(s'),1\}}.
	\end{align*}
Hence,
\[d_{C(Y)}((y,t),(y',t'))~=~\inf_{s\geq \max\{t,t'\}} \left\{2s-t-t'+\tfrac{d_Y(y,y')}{\max\{\rho(s),1\}}\right\}.\qedhere\]
\end{proof}

For $t \geq 0$, define the function $\phi_t\colon [0,\infty)\to [0,\infty)$ by
\[
\phi_t(r) \colonequals \inf_{s\geq t} \left\{2(s-t)+\tfrac{r}{\max\{\rho(s),1\}}\right\} = \inf_{u \geq 0} \left\{2u+\tfrac{r}{\max\{\rho(u+t),1\}}\right\}.
\]
Note that $d_{C(Y)}((y,t), (y',t')) = \phi_{\max\{t, t'\}}(d_Y(y,y')) + |t -t'|$.

\begin{cor}
	The subspace $Y_t \colonequals Y\times\{t\}\subseteq C(Y)$ is isometric to $Y$ equipped with the metric $\phi_t \circ d_Y$.
\end{cor}

\begin{example}
Let $\rho(s) = e^s$.  For this $\rho$, an elementary calculus exercise reveals that
\[
\phi_t(r)=
\begin{cases}
  e^{-t} r  & \text{if  $0 \leq r < 2 e^t$,}\\
 2\left( \ln(r/2) -t\right) + 2 & \text{if  $r \geq 2 e^t$.}
\end{cases}
\]
\end{example}

\begin{prop}\label{prop:phiproperties}
	For all $t \geq0$,
the function $\phi_t$ has the following properties.
	\begin{enumerate}
		\item\label{prop:phi1} 
		If $t' \geq t$ then $\phi_{t'} \leq \phi_{t}$.

		\item\label{prop:phi2} 
		$\phi_t$ is strictly increasing.

		\item\label{prop:phi3} 
		$\lim_{r \to \infty} \phi_t(r) = \infty$,

		\item\label{prop:phi4} 
		For all  $r, r' \geq 0$, we have $ |  \phi_t(r') - \phi_t(r) | \leq |r' -r| / \max\{\rho(t),1\}$. That is, $\phi_t$ is $(1 /  \max\{\rho(t),1\})$-Lipschitz.

		\item\label{prop:phi5} 
		If $\rho$ is proper, then $\lim_{t\to\infty}\phi_t(r)=0$ for all $r\geq 0$.

		\item\label{prop:phi6} 
		$\phi_t$ is a homeomorphism.

		\item\label{prop:phi7} 
		$\phi_t$ is concave. That is,  $\phi_t(\lambda r + (1 - \lambda)r') \geq \lambda\phi_t(r) + (1 - \lambda)\phi_t(r')$ for all $r, r' \geq 0$
and $0 \leq \lambda \leq 1$.

		\item\label{prop:phi8} 
$\phi_t$ is subadditive. That is, $\phi_t(r +r') \leq  \phi_t(r) + \phi_t(r')$ for all $r, r' \geq 0$.
Furthermore, for any $M \geq 1$,  $\phi_t(M r) \leq  M\phi_t(r)$.

		\item\label{prop:phi9} 
		For all $\delta \geq 0$, $\phi_t \leq \phi_{t + \delta} \, + \, 2\delta$.
	\end{enumerate}
\end{prop}

\begin{proof}
	(\ref{prop:phi1}) Assume $t' \geq t$  and $r' \geq r \geq 0$.   Since $\rho$ is non-decreasing, we have for  $u \geq0$,
\[
2u +\tfrac{r}{\max\{\rho(u+t'),1\}}  \leq 2u +\tfrac{r'}{\max\{\rho(u+t),1\}}  
\]
and it follows that $\phi_{t'}(r) \leq \phi_t(r')$.  
Hence, $\phi_{t'} \leq \phi_{t}$  and $\phi_t$ is non-decreasing.

	(\ref{prop:phi2}) Assume $0 \leq x \leq y$ and $\phi_t(x) = \phi_t(y)$.
For each positive integer $n$ there exists $u_n \geq 0$ such that 
\[
\tfrac{1}{n} + \phi_t(y) >   2u_n +\tfrac{y}{\max\{\rho(u_n+t),1\}}. 
 \]
Let $C = \tfrac{1}{2}\left( 1 + \phi_t(y) \right)$. 
The above inequality implies that $u_n <  C$ for all $n$.
We also have that
\[
\phi_t(y) =  \phi_t(x)  \leq   2u_n +\tfrac{x}{\max\{\rho(u_n+t),1\}}. 
 \]
Thus, for all $n$,
\[
\tfrac{1}{n} >   \tfrac{y - x}{\max\{\rho(u_n+t),1\}} \geq  \tfrac{y - x}{\max\{\rho(C+t),1\}} \geq 0.
 \]
It follows that  $y - x = 0$ and so $x =y$, which shows that $\phi_t$ is strictly increasing.

	(\ref{prop:phi3}) Suppose $\lim_{r \to \infty} \phi_t(r) = \infty$ is false.
Since $\phi_t$ is an increasing function, $\phi_t$ is bounded;  that is,
there exists a $C > 0$ such that $\phi_t(r) \leq C$ for all $r \geq0$.
For each positive integer $n$ there exists $u_n \geq 0$ such that
\[
1 + \phi_t(n) >   2u_n +\tfrac{n}{\max\{\rho(u_n+t),1\}}. 
 \]
This implies that $u_n < C'= \tfrac{1}{2}(1+C)$ for all $n$. Hence,
\[
1 + C >   2u_n +\tfrac{n}{\max\{\rho(u_n+t),1\}} \geq  \tfrac{n}{\max\{\rho(C'+t),1\}} \geq 0
 \]
for all $n$, a contradiction.
 
	(\ref{prop:phi4}) Assume $r' \geq r \geq 0$. For $u \geq 0$,
	\begin{align*}
		2u  +\tfrac{r'}{\max\{\rho(u+t),1\}} & = 2u  +\tfrac{r}{\max\{\rho(u+t),1\}} +  \tfrac{r'-r}{\max\{\rho(u+t),1\}} \\
		& \leq 2u  +\tfrac{r}{\max\{\rho(u+t),1\}} +  \tfrac{r'-r}{\max\{\rho(t),1\}}.
	\end{align*}
Therefore, $\phi_t(r') \leq \phi_t(r) +  (r'-r) / \max\{\rho(t),1\}$.
Since  $0 \leq \phi_t(r') - \phi_t(r)$, it follows that
$ |  \phi_t(r') - \phi_t(r) | \leq |r' -r| / \max\{\rho(t),1\}$.

	(\ref{prop:phi5}) If $\rho$ is proper, then $\lim_{t\to\infty} 1 / \max\{\rho(t),1\} =0$ and so it follows from
        (\ref{prop:phi4}) that $\lim_{t\to\infty}\phi_t(r)=0$ for all $r\geq 0$.

	(\ref{prop:phi6}) The map $\phi_t$ is surjective since it is continuous, $\phi_t(0)=0$ and  $\lim_{r \to \infty} \phi_t(r) = \infty$.  
It is injective and open because  it is continuous and strictly increasing.  Hence  $\phi_t$ is a homeomorphism.

	(\ref{prop:phi7}) Let $r, r '\geq 0$ and $0 \leq \lambda \leq 1$.
Then
	\begin{align*}
		\phi_t(\lambda r+(1-\lambda)r')&=\inf_{u\geq 0} \left\{2u+\tfrac{\lambda r+(1-\lambda)r'}{\max\{\rho(u+t),1\}}\right\}\\
		&=\inf_{u\geq 0} \left\{\lambda \left(2u+\tfrac{r}{\max\{\rho(u+t),1\}}\right)+(1-\lambda)\left(2u+\tfrac{r'}{\max\{\rho(u+t),1\}}\right)\right\} \\
		&\geq \lambda\inf_{u\geq 0} \left\{2u+\tfrac{r}{\max\{\rho(u+t),1\}}\right\} + (1-\lambda)\inf_{u'\geq 0}\left\{2u'+\tfrac{r'}{\max\{\rho(u'+t),1\}}\right\}\\
		&=\lambda\phi_t(r)+(1-\lambda)\phi_t(r').
	\end{align*}
Thus, $\phi_t$ is concave.

	(\ref{prop:phi8}) 
It is well-known that any concave function $f$ on an interval containing $0$ and with $f(0)=0$ is subadditive 
and satisfies $f(Mx) \leq M f(x)$ for $M \geq 1$ and all $x$.  In particular, $\phi_t$ has the stated properties.

	(\ref{prop:phi9}) Let $\delta \geq 0$.  Then
	\begin{align*}
		\phi_{t+\delta}(r)&=\inf_{u\geq 0} \left\{2u+\tfrac{r}{\max\{\rho(u+t+\delta),1\}}\right\} \\
		&=\inf_{u\geq \delta} \left\{2u-2\delta+\tfrac{r}{\max\{\rho(u+t),1\}}\right\} \\
		&\geq\inf_{u\geq 0} \left\{2u-2\delta+\tfrac{r}{\max\{\rho(u+t),1\}}\right\} \\
		&=\phi_t(r)-2\delta. \qedhere
	\end{align*}
\end{proof}

\begin{cor}\label{cor:embedding}
	For all $t \geq 0$, the map $\theta_t\colon Y\to C(Y)$  given by $\theta_t(y)=(y,t)$ is $(1 /  \max\{\rho(t),1\})$-Lipschitz and a coarse embedding.
\end{cor}

\begin{proof} 
	By \cref{prop:phiproperties}, for all $y, y' \in Y$
\[
d_{C(Y)}(\theta_t(y), \, \theta_t(y')) = \phi_t(d_Y(y,y')) \leq ( 1 /  \max\{\rho(t),1\})d_Y(y,y').
\]
Also, $\phi_t$ is strictly increasing and a homeomorphism (and therefore proper).
\end{proof}

Observe that \cref{cor:embedding} implies that for any metric family $\mcY$, the map of families $\Theta\colon \mcY\to C(\mcY)$ given by $\Theta \colonequals \left\{ \theta_0\colon Y\to C(Y)  \right\}_{Y\in \mcY}$ is a coarse embedding. We prove the following very general extension result.

\begin{prop}\label{prop:extensiontocone}
	Let $\mcX$ be a metric family with a decomposition $X=X_0\cup X_1$ for every $X\in\mcX$. Let $\mcX_0=\{X_0\}_{X\in\mcX}$ and let $\mcY$ be any metric family.
If $F\colon \mcX_0\to \mcY$ is a coarse map then there exist a monotonically increasing function $\rho\colon[0,\infty)\to[0,\infty)$ and a coarse map $F'\colon \mcX\to C(\mcY)$ such that
$\Theta \circ F$ is close to $F'|_{\mcX_0}$. Here the cone $C(\mcY)$ is constructed using the function $\rho$.
\end{prop}

\begin{proof}
	Let $\rho'\colon[0,\infty)\to[0,\infty)$ be a monotonically increasing function such that
$d_Y(f(x),f(y)) \leq \rho'(d_X(x,y))$ for all $x,y\in X_0\in \mcX_0$ and $f\colon X \to Y$ in  $F$. Define $\rho(t) \colonequals \max\{\rho'(3t+2),1\}$. For each $X=X_0\cup X_1\in \mcX$, let $p\colon X_0\cup X_1\to X_0$ be a map with $d_X(p(x),x)\leq d_X(X_0,x)+1$ and define $f'\colon X_0\cup X_1\to C(Y)$ by $x\mapsto (f(p(x)),d_X(x,X_0))$. We have 
	\begin{align*}
		d_X(p(x),p(x')) & \leq d_X(x,x')+d_X(x,X_0)+d_X(x',X_0)+2 \\
                 &\leq 3\max\{d_X(x,x'), \, d_X(x,X_0), \, d_X(x',X_0)\}+2
	\end{align*}
and thus,
	\begin{align*}
		\tfrac{d_Y(f(p(x)),f(p(x')))}{\max\{\rho\left(\max\{d_X(x, \, X_0),d_X(x', \, X_0)\}\right),1\}} & \leq\tfrac{\rho'\left(3\max\{d_X(x,x'), \, d_X(x, \, X_0), \, d_X(x', \, X_0)\}+2\right)}{\max\{\rho'\left(3\max\{d_X(x,X_0), \, d_X(x',X_0)\}+2\right),1\}}\\
		& \leq \max\{\rho'(3d_X(x,x')+2),1\}\\
		& =\rho(d_X(x,x')).
	\end{align*}
Hence,
	\begin{align*}
		d_{C(Y)}(f'(x),f'(x'))&\leq | d_X(x,X_0) - d_X(x',X_0) |+\tfrac{d_Y(f(p(x)),f(p(x')))}{\max\{\rho\left(\max\{d_X(x,X_0),d_X(x', X_0)\}\right),1\}}\\
		&\leq d_X(x,x')+\rho(d_X(x,x'))
	\end{align*}
which shows that $f'$ is a coarse map. Observe that for $x \in X_0$,
\[
d_{C(Y)}(\theta_0 \circ f(x),f'(x)) = d_{C(Y)}((f(x),0), \, (f(p(x)),0) ) \leq d_X(x, p(x)) \leq 1,
\]
and so $\theta_0 \circ f$ and $f'|_{X_0}$ are close.
\end{proof}

Next, we analyze the large scale dimension theory of $C(\mcY)$ for a metric family~$\mcY$.

\begin{lemma}\label{lem:asdimANzero}
	Let $\mcY$ be a metric family and $\mcB$  be a family of subsets of $[0, \infty)$ with the Euclidean metric. Let ${\mcY}_{\mcB} = \left\{ Y \times B \subset C(Y)~|~{Y\in \mcY, B\in\mcB}\right\}$. If $\asdim_{{\rm AN}}(\mcY) < \infty$ and  $\asdim_{{\rm AN}}(\mcB) = 0$, then $\asdim_{{\rm AN}}({\mcY}_{\mcB}) \leq  \asdim_{{\rm AN}}({\mcY})$.
\end{lemma}

\begin{proof}
	Since $\asdim_{{\rm AN}}(\mcB) = 0$ there exist non-negative constants ${M'}$ and ${b'}$ such that for each $B\in\mcB$ and each $R > 0$ there exists a cover $\mcA_B$ of $B$ that is $R$-disjoint and $\mesh (\mcA_B) \leq {M'} R + {b'}$. Note that $\mcA_B$ must be a countable collection otherwise $[0,\infty)$ would have an uncountable discrete subset, an impossibility. Let $\mcA_B = \{ A_j ~|~  j=1, 2, \dots\}$  and for each $j \geq 1$, let $a_j = \inf A_j$. Since $\mcA_B$  is $R$-disjoint,  $|a_i - a_j| \geq R$ for $i \neq j$. Observe that $A_k\subseteq [a_k,a_k+M'R+b']$. 

	Let $n =\asdim_{{\rm AN}}({\mcY})$. There is an $M\geq 1$ and a $b \geq 0$ such that for each $Y \in \mcY$, $R>0$ and positive integer $k$, there is a cover $\mcU^k$ of $Y$ with the property: 
$\mcU^k = \mcU^k_0  \,  \cup \cdots \cup \,  \mcU^k_n$, where each collection $\mcU^k_j$ is $\phi_{a_k + {M'} R + {b'}}^{-1}(R)$-disjoint and $\mesh (\mcU^k) \leq M \phi_{a_k + {M'} R + {b'}}^{-1}(R) + b$. For $0 \leq j \leq n$, let ${\mcW}_j = \{ U \times A_k ~|~ U \in \mcU^k_j, \, k \geq1\}$. Observe that $\mcW = {\mcW}_0 \cup \cdots \cup {\mcW}_n$ is a cover of $Y \times B$ and each ${\mcW}_j$ is $R$-disjoint. 

	Let $(y,t), (y', t') \in U \times A_k \in \mcW_j$.  Then, making use of \cref{prop:phiproperties}, we have
	\begin{align*}
		d_{C(Y)}( (y,t), (y', t')) &= \phi_{\max\{t,t'\}}(d_Y(y,y')) + |t -t'| \\
		& \leq \phi_{a_k}(d_Y(y,y')) + {M'} R + {b'} \\
                 & \leq \phi_{a_k + {M'} R + {b'}}(d_Y(y,y')) + 2({M'} R + {b'}) +{M'} R + {b'} \\
                 & \leq \phi_{a_k + {M'} R + {b'}}\left(M \phi_{a_k + {M'} R + {b'}}^{-1}(R) +b\right) + 3{M'}R  + 3b' \\
                 & \leq M R+ \phi_{a_k + {M'} R + {b'}}(b) + 3{M'}R  + 3b' \\
                 & \leq (M+3{M'} )R +   \phi_{0}(b) +3 b'.
	\end{align*}
Hence, $\mesh(\mcW)  \leq (M+3{M'} )R +   \phi_{0}(b) +3 b'$
and so
\[
\asdim_{{\rm AN}}({\mcY}_B) \leq  n =\asdim_{{\rm AN}}({\mcY}).  \qedhere
\]
\end{proof}

\begin{theorem}\label{thm:coneasdimAN}
	Let $\mcY$ be a metric family with $\asdim_{{\rm AN}}(\mcY) < \infty$. 
Then
\[
\asdim_{{\rm AN}}(C(\mcY)) \leq  \asdim_{{\rm AN}}(\mcY) +1.
\]
\end{theorem}

\begin{proof}
	Let $p\colon C(\mcY)\to \{ [0,\infty) \}$ be $\left\{ p_Y \colon C(Y) \to [0,\infty)\right\}_{Y \in \mcY}$, where $p_Y(y,t) = t$ for $(y,t) \in C(Y)$. The map $p$ is $1$-Lipschitz since
	\begin{align*}
		|p_Y(y,t) -  p_Y(y',t')| &= |t -t'| \\
		& \leq \phi_{\max\{t, t' \} }(d_Y(y,y')) + |t-t'| \\
		& = d_{C(Y)}((y,t), (y',t')).
	\end{align*}
Recall that ${\mcY}_{\mcB} = \left\{ Y \times B \subset C(Y)~|~{Y\in \mcY, B\in\mcB}\right\}$, where $\mcB$ is a family of subsets of $[0,\infty)$. Note that for any $Y \in \mcY$ and any $A \subset Y \times [0,\infty)$, we have $A \subset Y \times p_Y(A)$ and $p_Y(A) =  p_Y(Y \times p_Y(A))$. Hence, any subfamily $\mcA$ of $C(\mcY)$ is also a subfamily of $\mcY_{p(\mcA)}$ and 
$\asdim_{{\rm AN}}(\mcA) \leq \asdim_{{\rm AN}}(\mcY_{p(\mcA)})$. 

Let $\fC$ denote the collection of all families, $\mcB$, of subsets of $[0,\infty)$. We have
\[
\asdim_{{\rm AN}}(p) = \sup\{ \asdim_{{\rm AN}}(\mcY_{\mcB}) ~|~  \mcB \in \fC \text{ and } \asdim_{{\rm AN}}(\mcB)=0\}.
\]
By \cref{lem:asdimANzero}, $\asdim_{{\rm AN}}(p) \leq \asdim_{{\rm AN}}(\mcY)$.
Note that $\asdim_{{\rm AN}}([0,\infty)) =1$. By \cref{thm:BDLMtheorem},
\[
\asdim_{{\rm AN}}(C(\mcY)) \leq \asdim_{{\rm AN}}(p) + \asdim_{{\rm AN}}([0,\infty))
\leq \asdim_{{\rm AN}}(\mcY) + 1. \qedhere
\]
\end{proof}

We can now prove the main theorem of this section.

\begin{proof}[Proof of {\rm \cref{thm:extension}}]
Let $n = \asdim(\mcY)$. By \cite[Proof of Theorem 4.1]{fdc}, there is a coarse embedding $\Psi \colon \mcY\to \mcT$, where each $T\in \mcT$ is a metric product of $n+1$ Gromov 0-hyperbolic spaces. Recall that $\Theta \colon \mcT \to C(\mcT)$
is a coarse embedding (see the discussion following \cref{cor:embedding}) and thus, so is $\Theta \circ \Psi$. By \cref{lem:productof0hyp}, $\asdim_{{\rm AN}} (\mcT)  \leq n+1$. \cref{thm:coneasdimAN} implies that $\asdim_{{\rm AN}} (C(\mcT))  \leq n + 2$ and so $\asdim (C(\mcT)) \leq \asdim_{{\rm AN}} (C(\mcT))   \leq n + 2$
and, in particular, $\asdim (C(\mcT))$ is finite. That is, $C(\mcT) \in \fR_1$.
The conclusion of \cref{thm:extension} now follows by applying \cref{prop:extensiontocone} to $\Psi \circ F\colon \mcX_0\to \mcT$.
\end{proof}


\section{Permanence properties}\label{sec:inheritance}

In this section we prove several permanence properties for regular FDC.  We begin by showing that regular FDC satisfies Fibering Permanence (\cref{def:fibering}) and then show that all of the other permanence properties are satisfied for every collection of metric families that satisfies Fibering Permanence and contains all metric families with finite asymptotic dimension. We then use these permanence properties to show that several important classes of groups have regular FDC. Finally, we show that regular FDC possesses Finite Quotient Permanence (\cref{defi:quotients}), a property that FDC is not known to possess.

\begin{defi} \label{def:fibering}
A collection of metric families, $\fC$, satisfies \emph{Fibering Permanence} if the following holds. Let $F\colon \mcX\to\mcY$ be a coarse map of metric families. If $\mcY\in\fC$ and for every bounded subfamily $\mcZ$ of $\mcY$ the inverse image $F^{-1}(\mcZ) \in \fC$, then $\mcX\in\fC$.
\end{defi}

\begin{thm}\label{thm:fibering}
Regular FDC satisfies Fibering Permanence.
\end{thm}

\begin{proof}
	Let $\alpha$ be an ordinal number and let $P(\alpha)$ be the following statement: If there exists a $\mcY \in \fR_\alpha$ and a coarse map $F:\mcX \to \mcY$ such that for every bounded subfamily $\mcZ$ of $\mcY$, $F^{-1}(\mcZ)\in \fR$, then $\mcX \in \fR$. We prove by induction on $\alpha$ that $P(\alpha)$ is true for all $\alpha$, thereby establishing the theorem.

If $\mcY$ is bounded and $F:\mcX \to \mcY$ is a coarse map such that for every bounded subfamily $\mcZ$ of $\mcY$, $F^{-1}(\mcZ)\in \fR$, then $\mcX=F^{-1}(\mcY)\in \fR$. Thus, $P(0)$ is true. Now assume that $P(\beta)$ is true for every $\beta < \alpha$, and assume there exists a $\mcY \in \fR_\alpha$ and a coarse map $F:\mcX \to \mcY$ such that for every bounded subfamily $\mcZ$ in $\mcY$, $F^{-1}(\mcZ)\in \fR$. 

Since $\mcY \in \fR_\alpha$, there exists a metric family $\mcW$ with $\asdim(\mcW) < \infty$ and a coarse map $G:\mcY \to \mcW$ such that for every bounded subfamily $\mcB$ in $\mcW$, there exists a $\beta < \alpha$ such that $G^{-1}(\mcB)\in \fR_\beta$. Consider the coarse map $G\circ F : \mcX \to \mcW$ and let $\mcB$ be a bounded subfamily of $\mcW$. By \cref{prop:ordinals}, it will follow that $\mcX \in \fR$ if we can show that $(G\circ F)^{-1}(\mcB) \in \fR$. Note that the restriction of $F$ to $(G\circ F)^{-1}(\mcB)$, denoted $F\colon (G\circ F)^{-1}(\mcB)\to G^{-1}(\mcB)$, is a coarse map and that $G^{-1}(\mcB)\in \fR_\beta$ for some $\beta < \alpha$. Now let $\mcZ$ be a bounded subfamily of $G^{-1}(\mcB)$ (which is a subfamily of  $\mcY$). By assumption, $F^{-1}(\mcZ)\in \fR$. Applying the induction hypothesis $P(\beta)$ to $(G\circ F)^{-1}(\mcB)$, we have that $(G\circ F)^{-1}(\mcB) \in \fR$.
\end{proof}

Since regular FDC satisfies Fibering Permanence and the definition of regular decomposition is a special case of fibering, we immediately obtain the following theorem.

\begin{thm}\label{thm:smallest_fibering}
The collection of metric families with regular FDC is the smallest collection of metric families that contains all families with finite asymptotic dimension and satisfies Fibering Permanence.
\end{thm}

Fibering Permanence is a strong property for a collection of metric families, $\fC$, to possess. As we show below, if $\fC$ satisfies Fibering Permanence and contains all bounded metric families, then $\fC$ must also satisfy Coarse Permanence (\cref{def:coarse_permanence}) and Finite Amalgamation Permanence (\cref{def:amalgam}). If $\fC$ additionally contains all metric families with finite asymptotic dimension, then $\fC$ must also satisfy Finite Union Permanence (\cref{def:finunion}), Union Permanence (\cref{def:union}), and Limit Permanence (\cref{def:limit}).

\begin{thm}\label{thm:coarseinv-general}
	Let $\fC$ be a collection of metric families that satisfies Fibering Permanence and contains all bounded metric families. Then $\fC$ satisfies Coarse Permanence.
\end{thm}

\begin{proof}
	Let $F\colon \mcX\to \mcY$ be a coarse embedding with $\mcY\in \fC$. Then for every bounded subfamily $\mcB$ of $\mcY$ the inverse image $F^{-1}(\mcB)$ is again bounded. Thus, $\mcX\in\fC$ by Fibering Permanence.
\end{proof}

Note that we are not able to combine \cref{thm:coarseinv-general} with \cref{thm:fibering} to deduce Coarse Permanence for regular FDC, since our proof of Fibering Permanence for regular FDC uses \cref{prop:ordinals}, which relies on \cref{thm:coarseinv}.

\begin{defi}\label{def:amalgam}
	A collection of metric families, $\fC$, satisfies \emph{Finite Amalgamation Permanence} if the following holds. If $\mcX=\bigcup_{i=1}^n\mcX_i$ and each $\mcX_i \in \fC$, then $\mcX \in \fC$.
\end{defi}

\begin{thm}\label{thm:amalgamation-general}
	Let $\fC$ be a collection of metric families that satisfies Fibering Permanence and contains all bounded metric families. Then $\fC$ satisfies Finite Amalgamation Permanence.
\end{thm}

\begin{proof}
It suffices to prove Finite Amalgamation Permanence in the case $n=2$. Let $\mcX_1,\mcX_2\in\fC$ be given and define $\mcX_1\times\mcX_2:=\{X_1\times X_2\mid X_i\in\mcX_i\}$. Define a coarse map $P\colon \mcX_1\times \mcX_2\to \mcX_1$ by projection onto the first factor. For every bounded subfamily $\mcB$ of $\mcX_1$ the  inverse image $P^{-1}(\mcB)$ is coarsely equivalent to $\mcX_2$. Thus, Fibering Permanence implies that $\mcX_1\times \mcX_2\in\fC$. By fixing $Z \in \mcX_1$, $Y \in \mcX_2$, $z \in Z$ and $y\in Y$, we obtain a coarse embedding $F\colon \mcX_1\cup\mcX_2\to \mcX_1\times \mcX_2$ by sending $X_1\in\mcX_1$ to $X_1\times Y$ via $x_1\mapsto (x_1,y)$ and sending $X_2\in\mcX_2$ to $Z\times X_2$ via $x_2\mapsto (z,x_2)$. Therefore, since $\fC$ satisfies Coarse Permanence (by \cref{thm:coarseinv-general}), $\mcX_1\cup\mcX_2\in\fC$.
\end{proof}

\begin{defi}\label{def:finunion}
	A collection of metric families, $\fC$, satisfies \emph{Finite Union Permanence} if the following holds. For $n \in \bbN$, let $\mcX_1, \dots, \mcX_n\in\fC$ and let $\mcX$ be a metric family. If for each $X\in \mcX$ there exist $X_i\in \mcX_i$, $1\leq i \leq n$, such that $X=\bigcup_{i=1}^n X_i$, then $\mcX \in \fC$.
\end{defi}

\begin{thm}\label{thm:finunion2}
	Let $\fC$ be a collection of metric families that satisfies Fibering Permanence and contains all metric families with finite asymptotic dimension. Then $\fC$ satisfies Finite Union Permanence.
\end{thm}

\begin{proof}
	It suffices to prove the case $n=2$. Let $\mcX, \mcX_1, \mcX_2$ be metric families, where $\mcX_1, \mcX_2 \in \fC$. Assume that for each $X\in\mcX$ there exist $X_1\in \mcX_1$ and $X_2\in \mcX_2$ such that $X=X_1\cup X_2$. Define a map $f_X: X \to \bbR$ by $f_X(x)= d_X(x, X_2)-d_X(x,X_1)$, where $d_X$ is the metric on $X$. Notice that $f_X$ is a 2-Lipschitz map. Therefore, $F=\{f_X\}: \mcX \to \{\bbR\}$ is a coarse map to a metric family with asymptotic dimension 1. Let $\mcU$ be a bounded subfamily of $\{\bbR\}$ (i.e., a collection of uniformly bounded subsets of $\bbR$). If we show that $F^{-1}(\mcU) \in \fC$, then it will follow from Fibering Permanence that $\mcX\in\fC$.
	
	Let $D=\sup\{\diam(U) : U\in\mcU \}$. Then $\mcU$ is a subfamily of the family $\mcV = \big\{ (-\infty, D], [-D, \infty) \big\}$. Notice that $f_X^{-1}\big((-\infty, D]\big) = B_D(X_2)$ and $f_X^{-1}\big([-D, \infty)\big) = B_D(X_1)$, where $B_D(X_i)$ denotes the $D$-neighborhood of $X_i$. Thus, $F^{-1}(\mcV)$ is coarsely equivalent to the amalgamation $\mcX_1 \cup \mcX_2$, which lies in $\fC$ by Finite Amalgamation Permanence (\cref{thm:amalgamation-general}). Since $F^{-1}(\mcU)$ is a subfamily of $F^{-1}(\mcV)$, it also lies in $\fC$ by Coarse Permanence (\cref{thm:coarseinv-general}). This completes the proof.
\end{proof}

\begin{defi}\label{def:union}
	A collection of metric families, $\fC$, satisfies \emph{Union Permanence}\footnote{Note that in \cite{guentner} a collection is said to satisfy Union Permanence if it satisfies the above definition of Union Permanence, as well as Finite Union Permanence.} if the following holds. Let $\mcX=\{X_i\}_{i\in I}$ be a metric family in which each $X_i\in \mcX$ is expressed as a union of metric subspaces $X_i=\bigcup_{i\in J_i}X_{ij}$. If $\{X_{ij}\mid i\in I, j\in J_i\} \in \fC$ and for each $r>0$ there exist subspaces $Y_i(r)\subseteq X_i$ such that $\{Y_i(r)\}_{i\in I} \in \fC$ and $\{Z_{ij}(r)=X_{ij}\setminus Y_i(r)~|~j\in J_i\}$ is an $r$-disjoint collection for each $i\in I$, then $\mcX\in\fC$.
\end{defi}

\begin{thm}\label{thm:union}
	Let $\fC$ be a collection of metric families that satisfies Fibering Permanence and contains all metric spaces with finite asymptotic dimension. Then $\fC$ satisfies Union Permanence.
\end{thm}

\begin{proof}
	Note that, by the previous theorems, $\fC$ satisfies Coarse Permanence, Finite Amalgamation Permanence, and Finite Union Permanence.
	
	Let $\mcX=\{X_i\}_{i\in I}$ be a metric family in which each $X_i\in \mcX$ is expressed as a union of metric subspaces $X_i=\bigcup_{i\in J_i}X_{ij}$. Suppose that $\{X_{ij}\mid i\in I, j\in J_i\} \in \fC$ and for each $r>0$ there exist subspaces $Y_i(r)\subseteq X_i$ such that $\{Y_i(r)\}_{i\in I} \in \fC$ and $\{Z_{ij}(r)=X_{ij}\setminus Y_i(r)~|~j\in J_i\}$ is an $r$-disjoint collection for each $i\in I$. Let $Y'_i(1)=Y_i(1)$ and for $n\geq 2$ define $Y'_i(n) \colonequals B_{n-1}(Y'_i(n-1))\cup Y_i(n)$. (If $\bigcup_{r\in\bbN}Y_i(r)$ is empty, then choose $Y_i'(1)$ to be a single point.) Notice that for every $i\in I$ and every $n\in \bbN$, $Y'_i(n) \subseteq Y'_i(n+1)$ and the collection $\{Z'_{ij}(n) \colonequals X_{ij}\setminus Y'_i(n)~|~j\in J_i\}$ is $n$-disjoint. Also, for each $i\in I$, $\bigcup_{n=1}^{\infty} Y_i'(n) = X_i$. Since $\{Y_i(n)\}_{i\in I} \in \fC$ for each $n\in \bbN$, Coarse Permanence and Finite Union Permanence imply that $ \{Y_i'(n)\}_{i\in I}=\{B_{n-1}(Y_i'(n-1))\cup Y_i(n)\}_{i\in I}\in\fC$ for each $n\in \bbN$.

For each $i\in I$ and $j\in J_i$, let $L_{ij}$ denote the ray $[0,\infty)$ considered as a graph with vertex set $\bbN \cup \{0\}$. For each $i\in I$, let $T_i$ be the rooted tree obtained from $\bigsqcup_{j\in J_i}L_{ij}$ by identifying the set $\{0\in L_{ij} \mid j\in J_i\}$ to one point $p_i$, the root of $T_i$. For every $i$, define a map $f_i\colon X_i\to T_i$ inductively as follows. Map all of $Y'_i(1)$ to the root $p_i \in T_i$. Now assume the map is defined for $Y_i'(n)$ and let $y\in Y_i'(n+1)\setminus Y_i'(n)$ be given. Define $f_i(y)=n\in L_{ij}\subset T_i$, where $y\in X_{ij}$. This is well-defined because if $y\in X_{ij} \cap X_{ik}$ for distinct $j,k\in J_i$, then $y\in Z'_{ij}(n) \cap Z'_{ik}(n)$, which is impossible since $Z'_{ij}(n)$ and $Z'_{ik}(n)$ are disjoint.

We want to show that $F=\{f_i\colon X_i\to T_i\}_{i\in I}$ is a coarse map. Let $x,y\in X_i$ be given. Then there exist $n,m\in \bbN$ with $n-1\leq d_{X_i}(x,y)\leq n$ and $f_i(x)=m\in L_{ij}$. We can assume without loss of generality that $m\leq f_i(y)$. Then $x\in X_{ij} \cap Y_i'(m+1)\setminus Y_i'(m)$ and $y\notin Y_i'(m)$. We have $y\in B_n(Y'_i(m+1))\subseteq B_{N}(Y_i'(N))\subseteq Y_i'(N+1)$, where $N=\max\{n,m+1\}$. If $N=m+1$, then $y\in X_{ij}$ by the $N$-disjointness of $\{Z'_{ij}(N)~|~j\in J_i\}$, and so $f_i(y)\in [m,N]=[m,m+1]\subseteq L_{ij}$; that is, $d_{T_i}(f_i(x),f_i(y))\leq 1$. If $n>m+1$, then $f_i(y)\in [m,N]=[m,n]\subseteq L_{ik}$ for some $k\in J_i$, and $d(f_i(x),f_i(y))\leq n+m\leq 2n \leq 2d_{X_i}(x,y)+2$. Thus, $F\colon \mcX\to\{T_i\}_{i\in I}$ is a coarse map.

Since each $T_i$ is a tree, \cref{lem:productof0hyp} implies that the family $\{T_i\}_{i\in I}$ has asymptotic dimension at most one.
Let $\mcU$ be a bounded subfamily of $\{T_i\}_{i\in I}$ and let $m$ be an integer greater than $\sup\{\diam(U) \mid U\in\mcU \}$. 
Then $\mcU$ is a subfamily of
\[
\mcV =\{B_{2m}(p_i)\subseteq T_i~|~i\in I\}\cup \{L_{ij}\setminus [0,m]~|~i\in I, j\in J_i\}.
\]
We have that $F^{-1}(\mcV)$ is a subfamily of $\{Y'_i(2m+1)~|~i\in I\}\cup\{X_{ij}\mid i\in I, j\in J_i\}$. Since $\{X_{ij}\mid i\in I, j\in J_i\}$ and $\{Y'_i(2m+1)~|~i\in I\}$ lie in $\fC$, Finite Amalgamation Permanence implies that $F^{-1}(\mcV)\in\fC$. By Coarse Permanence, the subfamily $F^{-1}(\mcU)$ also lies in $\fC$. Thus, Fibering Permanence implies that $\mcX\in\fC$.
\end{proof}

The next permanence property we establish is Limit Permanence, as defined by Guentner \cite{guentner}.

\begin{defi}\label{def:limit}
	A collection of metric families, $\fC$, satisfies \emph{Limit Permanence} if the following holds. Let $\mcX=\{X_i\}_{i\in I}$ be a metric family. If for every $R>0$ there exists an indexing set $J_R$ and, for each $i\in I$, an $R$-disjoint decomposition
\[X_i=\bigsqcup_{R\text{-disjoint}}\{Y_{ij}~|~j\in J_R\}\]
such that $\{Y_{ij}\mid i\in I, j\in J_R\} \in \fC$, then $\mcX \in \fC$.
\end{defi}

\begin{thm}\label{thm:limit_perm}
	Let $\fC$ be a collection of metric families that satisfies Fibering Permanence and contains all metric families with finite asymptotic dimension. Then $\fC$ satisfies Limit Permanence.
\end{thm}

\begin{proof}
	 Let $\mcX=\{X_i\}_{i\in I}$ be a metric family. Suppose that for every $R>0$ there exist an indexing set $J_R$ and $R$-disjoint decompositions
	 \[X_i=\bigsqcup_{R\text{-disjoint}}\{Y_{ij}~|~j\in J_R\},\]
	 such that the family $\{Y_{ij}\mid i\in I, j\in J_R\}$ lies in $\fC$.
	
	For each $i\in I$, let $X_i'$ be the space with the same underlying set as $X_i$ equipped with the metric $d'_i$ defined as follows.  For $x,y\in X_i$,  $d'_i(x,y)=0$ if $x=y$,  and if $x \neq y$, then
	\[d'_i(x,y)=\inf\{\max\{1,d_i(x_m,x_{m+1})\mid m\in \bbN\}\mid x_m\in X_i: \exists n\in\bbN, x_0=x, x_n=y\}.\]
	This metric can be viewed as follows. We call two points $r$-connected if one can be reached from the other by any number of jumps of length at most $r$. In the above metric the distance between two points is the smallest $r$ such that the two points are $r$-connected, with the exception that two different points have distance at least one. The latter is only necessary to obtain a metric instead of a pseudo-metric.
	
	The metric $d'_i$ is an ultrametric; that is, it satisfies $d_i'(x,y)\leq\max\{d_i(x,z),d_i(z,y)\}$ for all $x,y,z\in X_i$. This ultrametric is coarsely equivalent to the pseudo-ultrametric constructed from a metric in \cite[Lemma 8]{lemin}.
Note that $d_i'(x,y)\leq \max\{d_i(x,y),1\}$. Thus, $F=\{\id_i\colon X_i\to X_i'\}_{i\in I}$, the map of metric families consisting of the identity functions on the underlying sets, is coarse. Furthermore, since for every $r>0$, two $r$-balls in $X_i'$ will either coincide or be disjoint, it is straightforward to show using \cref{prop:Lebesgue} that $\asdim(\mcX')=0$. Let $\mcB_r=\{B_r(x)\mid i\in I, x\in X'_i\}$ be the subfamily of all $r$-balls in $\mcX'$. Then, for each $R>0$, the family $F^{-1}(\mcB_R)$ is a subfamily of $\{Y_{ij}\mid i\in I, j\in J_R\}$, and so $F^{-1}(\mcB_R)\in \fC$ by \cref{thm:coarseinv-general}. Hence, by Fibering Permanence, $\mcX\in\fC$.
\end{proof}

We now have the following corollary.

\begin{cor}\label{cor:rFDC_permanence}
	Regular FDC satisfies Finite Amalgamation Permanence, Finite Union Permanence, Union Permanence, and Limit Permanence.
\end{cor}

In \cref{thm:coarseinv} we showed that $\fR_\alpha$ satisfies Coarse Permanence for every ordinal $\alpha$. We now show that $\fR_\alpha$ also satisfies Finite Amalgamation Permanence and Finite Union Permanence. The fact that $\fR_\alpha$ satisfies Finite Union Permanence will play an important role in the proof that regular FDC satisfies Finite Quotient Permanence (\cref{thm:quotients1}).

\begin{thm}\label{thm:finunion1}
	The collection $\fR_\alpha$ satisfies Finite Amalgamation Permanence for every ordinal $\alpha$.
\end{thm}

\begin{proof}
	We prove the theorem by induction on $\alpha$. Clearly, if each $\mcX_i\in \fR_0=\fB$, then also $\mcX\in \fR_0$. Now assume the statement holds for all $\beta<\alpha$. Assume that $\mcX=\bigcup_{i=1}^n\mcX_i$, where each $\mcX_i$ is in $\fR_\alpha$. Then, for each $i$ there is a $\mcY_i\in\fR_1$ and a coarse map $F_i\colon \mcX_i\to \mcY_i$ such that the inverse image of a bounded subfamily lies in $\fR_{\beta_i}$ for some $\beta_i<\alpha$. Note that the family $\mcY=\bigcup_{i=1}^n\mcY_i$ lies in $\fR_1$. The inverse image $F^{-1}(\mcB)$ of a bounded subfamily $\mcB$ of $\mcY$ under the canonical map $F\colon \mcX\to \mcY$ is a finite union of metric families lying in $\fR_\beta$, where $\beta=\max_i \beta_i < \alpha$. Therefore, $F^{-1}(\mcB)\in\fR_\beta$ by the induction hypothesis. Hence, $\mcX\in\fR_\alpha$. 
\end{proof}

\begin{thm}\label{thm:finunion2alpha}
	The collection $\fR_\alpha$ satisfies Finite Union Permanence for every ordinal $\alpha$.
\end{thm}

\begin{proof}
	Let $n \in \bbN$, and let $\mcX, \mcX_1, \dots, \mcX_n$ be metric families, where each $\mcX_i\in \fR_\alpha$. Assume that for each $X\in \mcX$ there exist $X_i\in \mcX_i$, $1\leq i \leq n$, such that $X=\bigcup_{i=1}^n X_i$. We prove that $\mcX\in\fR_\alpha$ by induction on $\alpha$. 
	
	The case $\alpha=0$ is trivial. (The case $\alpha=1$ is the fact that finite asymptotic dimension satisfies Finite Union Permanence. The proof of this fact for metric spaces, which can be found in \cite[Corollary 26]{bell-dranish-asdim}, immediately generalizes to metric families.) Suppose that the statement holds for all $\beta<\alpha$, and assume that $\mcX_i\in \fR_\alpha$ for every $i=1,\ldots,n$. Then there are coarse maps $F_i:\mcX_i\to \mcY_i$, where $\asdim\mcY_i<\infty$, such that the inverse images of bounded subfamilies $\mcB_i$ of $\mcY_i$ under $F_i$ lie in $\fR_{\beta_i}$, with $\beta_i<\alpha$. By \cref{cor:extension}, there exists a metric family $\mcY'$ with $\asdim\mcY'<\infty$ and a coarse map $F\colon\mcX\to \mcY'$  such that for every bounded subfamily $\mcB$ of $\mcY'$ there are bounded subfamilies $\mcB_i$ of $\mcY_i$ such that the inverse image $F^{-1}(\mcB)$ coarsely embeds into $\{\bigcup_{i=1}^nC_i\mid C_i\in F_i^{-1}(\mcB_i)\}$. Therefore, by the induction assumption, $F^{-1}(\mcB)$ lies in $\fR_{\beta}$, where $\beta=\max_i \beta_i < \alpha$. Hence, $\mcX\in\fR_\alpha$.
\end{proof}

Given a group $G$ together with a finite symmetric generating set $S \subset G$, the {\it length} of $g \in G$ with respect to $S$ is the non-negative integer
$| g |_S = \min\{ n ~|~ g = s_1 s_2 \cdots s_n, \, s_j \in S\}$.
The corresponding left-invariant {\it word metric} on $G$ is given by $d_S(g,h) = |g^{-1}h|_S$.
Any two such finite generating sets for $G$ yield quasi-isometric metric spaces.
More generally, a countable group $G$ admits a proper, left-invariant metric that is unique up to coarse equivalence.
Hence, asymptotic dimension, FDC, and regular FDC are well-defined for countable groups, and in the following we will assume all groups to be countable.

Since regular FDC satisfies Coarse Permanence, Fibering Permanence, Finite Union Permanence, Union Permanence, and Limit Permanence
(\cref{thm:coarseinv}, \cref{thm:fibering} and \cref{cor:rFDC_permanence}), the following permanence properties for groups now follow from \cite[Theorems 7.2.1, 7.2.3, 7.2.5 and Corollary 7.2.4]{guentner}.

\begin{cor}\label{cor:direct_union}
	If $G$ is a (countable) direct union of subgroups that each have regular FDC, then $G$ has regular FDC. In particular, a countable discrete group has regular FDC if and only if its finitely generated subgroups do.
\end{cor}

\begin{cor}
	If $G$ acts on a locally finite space $X$ with regular FDC and there exists an $x\in X$ for which the stabilizer subgroup $G_x$ has regular FDC, then $G$ has regular FDC.
\end{cor}

\begin{cor}\label{cor:groupextension}
The class of groups with
regular FDC is closed under group extensions.
\end{cor}

\begin{cor}
The class of groups with
regular FDC is closed under the formation of free products (with amalgamation).	
\end{cor}

\begin{cor}
	A group acting on a tree has regular FDC if and only if all of the vertex stabilizers have regular FDC.
\end{cor}

Since regular FDC satisfies all of the permanence properties that FDC does, the proofs that the following groups have regular FDC are precisely the same as in \cite{rigidity} and \cite{fdc}.

\begin{theorem}\label{thm:group_examples}
	The following classes of groups have regular FDC:
	\begin{enumerate}

		\item\label{item1:thm:group_example} Elementary amenable groups.
		
		\item Countable subgroups of $GL_n(R)$, where $R$ is any commutative ring with unit.
		
		\item Countable subgroups of virtually connected Lie groups.
	\end{enumerate}
\end{theorem}

\begin{proof}
	The class of elementary amenable groups is the smallest class of groups that contains all finite groups and all finitely generated abelian groups, and is closed under extensions and countable direct unions. Since all finite groups and all finitely generated abelian groups have finite asymptotic dimension, they have regular FDC. Therefore, elementary amenable groups have regular FDC by Corollaries~\ref{cor:direct_union} and~\ref{cor:groupextension}.

By \cref{cor:direct_union}, to prove that countable subgroups of $GL_n(R)$ have regular FDC it suffices to show that all finitely generated subgroups of $GL_n(R)$ have regular FDC. When $R$ is a field, the proof makes use of the Fibering Theorem~\ref{thm:fibering} and is exactly the same as the proof of \cite[Theorem 3.1]{rigidity}, except that in the proof of \cite[Lemma 3.9]{rigidity} FDC has to be replaced by regular FDC. For the generalization to linear groups over commutative rings, the proof is analogous to the proof of \cite[Theorem 5.2.2]{fdc}, making use of the fact that nilpotent groups are elementary amenable and hence have regular FDC by \eqref{item1:thm:group_example} and the fact that regular FDC is closed under group extensions (\cref{cor:groupextension}).

A countable subgroup of a virtually connected Lie group has a finite index subgroup that is contained in a connected Lie group. A subgroup of a connected Lie group can be written as an extension with linear quotient and abelian kernel, and thus has regular FDC.
\end{proof}

\begin{example}\label{rem:iterated}
	Included in the class of elementary amenable groups are the iterated wreath products of the infinite cyclic group.
Let $G_0 = \bbZ$ and,  for $n \geq 1$, let $G_n = G_{n-1} \wr \, \bbZ$, the wreath product of $G_{n-1}$ with $\bbZ$.
There is a natural inclusion $G_n \hookrightarrow G_{n+1}$.  Define $G_\omega$ to be the direct union $\bigcup_{n \geq 0} G_n$.
By \cref{thm:group_examples}(1),
$G_n$, $n \geq 0$, and $G_\omega$ have regular FDC.
Note that
for $n > 1$ the group $G_n$ is {\it not} a linear group,  \cite[Corollary 15.1.5]{robinson}.
Furthermore, $G_\omega$  is not solvable because it
contains $G_n$ for every $n$ and $G_n$ has derived length $n+1$.
\end{example}

Since regular FDC satisfies Coarse Permanence, Fibering Permanence, Finite Union Permanence, and Union Permanence, it is an \emph{axiomatically extendable} property of metric families, as defined by Ramras and Ramsey~\cite[Definition 3.7]{ramras-ramsey}. Such a property is extendable to {\it relatively hyperbolic groups}~\cite[Theorem 3.9]{ramras-ramsey}. That is, we obtain the following corollary.

\begin{cor}
	If $G$ is relatively hyperbolic with respect to the peripheral subgroups $H_1,\ldots, H_n$, and each $H_i$ has regular FDC, then $G$ has regular FDC.
\end{cor}

Regular FDC also behaves well with respect to taking quotients by finite groups (\cref{thm:quotients1} below).
This permanence property is not known for FDC and was one of the main motivations for our introduction of regular FDC.

Let $X$ be a metric space and $F$ a finite group acting isometrically on $X$. Recall that on the quotient space $F\backslash X$ we use the following metric.
\[d(Fx,Fx') \colonequals \min_{h\in F}d_X(x,hx')\]

\begin{rem}\label{rem:finite-quotient}
	With this choice of metric the quotient map $q_{X}:X \to F\backslash X$ is contracting and is therefore a coarse map with control function equal to the identity on $[0,\infty)$. Furthermore, an $F$-equivariant coarse map $f:X \to Y$ with control function $\rho$ induces a coarse map $\bar{f}:F\backslash X \to F\backslash Y$ that also has control function equal to $\rho$ and makes the following diagram commute.
\[
\xymatrix{
	X \ar[r]^f \ar[d]_{q_{X}} & Y \ar[d]^{q_{Y}}
	\\
	F\backslash X \ar[r]^{\bar{f}} & F\backslash Y
}
\]
\end{rem}

\begin{defi}\label{defi:quotients}
A collection, $\fC$, satisfies \emph{Finite Quotient Permanence} if the following holds. If $\mcX\in\fC$ and $F$ is a finite group that acts isometrically on every $X\in\mcX$, then the family $F\backslash \mcX:=\{F\backslash X\mid X\in\mcX\}$ is also in $\fC$.
\end{defi}

\begin{prop}
	\label{prop:FQPfamily}
Let $\fC$ be a collection that satisfies Finite Quotient Permanence. Let $\mcX=\{X_i\}_{i\in I}$ be in $\fC$, and let $\mcF=\{F_i\}_{i\in I}$ be a collection of finite groups such that there exists an integer $M$ with $\sup_i|F_i|\leq M$ and such that $F_i$ acts isometrically on $X_i$, for each $i\in I$. Then $\mcF\backslash \mcX:=\{F_i\backslash X_i\mid i\in I\}$ is in $\fC$.
\end{prop}

\begin{proof}
Since the orders of the finite groups $F_i$ have a uniform bound, there are finitely many finite groups,  $G_1, \ldots, G_n$, such that each $F_i$ is isomorphic to one of the $G_j$'s. Let $F=\bigoplus_{j=1}^nG_j$, and define the action of $F$ on $X_i$ by choosing an isomorphism from one of the summands of $F$ to $F_i$ and letting the other summands act trivially. This action of $F$ on $X_i$ is isometric and $F\backslash \mcX=\mcF\backslash \mcX$. Thus, $\mcF\backslash \mcX \in \fC$ since $\fC$ satisfies Finite Quotient Permanence.
\end{proof}

\begin{rem}
It is straightforward to generalize the proof of \cref{prop:asdim_quotient} to show that if $\mcX=\{X_i\}_{i\in I}$ is a metric family with asymptotic dimension at most $n$ and $\mcF=\{F_i\}_{i\in I}$ is a collection of finite groups such that each $F_i$ acts isometrically on $X_i$ and there exists an integer $M$ with $\sup_i|F_i|\leq M$, then $\asdim(\mcF\backslash \mcX) \leq M(n+1)-1$.
\end{rem}

\cref{prop:asdim_quotient} tells us that finite asymptotic dimension satisfies Finite Quotient Permanence.  The next theorem can be thought of as a generalization of this fact. While it is not hard to prove that weak FDC satisfies Finite Quotient Permanence, we do not know whether it holds for FDC.
Note that Finite Quotient Permanence is the only permanence property that we cannot derive from Fibering Permanence.

\begin{thm}\label{thm:quotients1}
The collection $\fR_\alpha$ satisfies Finite Quotient Permanence, for every ordinal $\alpha$. In particular, $\fR$ satisfies Finite Quotient Permanence.
\end{thm}

\begin{proof}	
We prove the theorem by induction on $\alpha$. The case $\alpha=0$ is clear, since if $\mcX$ is bounded, then so is $F\backslash\mcX$ for any finite group $F$ acting isometrically on every $X \in \mcX$. 

Now assume that the statement is true for all $\beta <\alpha$. Let $\mcX \in \fR_\alpha$ be given, and let $F$ be a finite group that acts isometrically on every $X \in \mcX$. By definition, there is a $\mcY\in\fR_1$ and a coarse map $G\colon \mcX\to \mcY$ such that the inverse image of each bounded subfamily of $\mcY$ lies in $\fR_\beta$ for some $\beta<\alpha$. Since $G$ is coarse, there is a control function $\rho\colon [0,\infty)\to[0,\infty)$ with $d_Y(g(x),g(x'))\leq \rho(d_X(x,x'))$ for all $g\colon X\to Y$ in $G$ and all $x,x'\in X$. 

Let $g\colon X\to Y$ be a map in $G$. (Recall that, by the definition of a coarse map of metric families, for each $X\in \mcX$ there is at least one $g\in G$ mapping $X$ to some $Y\in \mcY$.) 
Define a map $g_{F}\colon X\to \prod_{F}Y$ by $g_{F}(x)= \big(g(h^{-1}x)\big)_{h\in F}$. 
Note that $g_{F}$ is $F$-equivariant, where $F$ acts on $\prod_{F}Y$ by permuting factors. Equip $\prod_{F}Y$ with the $\ell^1$-metric,~$d^1$.
Then, $F$ acts isometrically on $\prod_{F}Y$, and
	\begin{align*}
		d^1\big(g_{F}(x),g_{F}(x')\big) &=\sum_{h\in F} d_{Y}\big(g(h^{-1}x),g(h^{-1}x')\big) \leq \sum_{h\in F} \rho(d_{X}(h^{-1}x,h^{-1}x')) \\
		&=|F|\cdot\rho(d_{X}(x,x')).
	\end{align*}
Hence, the map $g_{F}$ is coarse with control function $\rho'=|F|\cdot \rho$. Thus, $G_{F}=\{g_{F}\}$ defines a coarse map from $\mcX$ to the metric family $\mcY_{F}=\big\{\prod_{F}Y \mid Y\in\mcY\}$ with control function $\rho'$. Furthermore, since $g_{F}$ is $F$-equivariant, it induces a coarse map $\bar{g}_{F}:F\backslash X\to F\backslash\prod_{F}Y$ also with control function $\rho'$ (see \cref{rem:finite-quotient}). Therefore, $\overline{G}_{F}=\{\bar{g}_{F}\}$ defines a coarse map from $F\backslash\mcX$ to  $F\backslash \mcY_F$ with control function $\rho'$. Let $Q=\{q:X \to F\backslash X \mid X\in\mcX\}$ and $P=\{p:\prod_{F}Y \to F\backslash \prod_{F}Y \mid Y\in\mcY\}$, be the quotient maps. Then we have the following commutative diagram (see \cref{rem:finite-quotient}).
\[
\xymatrix{
	\mcX \ar[r]^{G_{F}} \ar[d]_Q & \mcY_{F} \ar[d]^P
	\\
	F\backslash\mcX \ar[r]^{\overline{G}_{F}} & F\backslash\mcY_F
}
\]

By \cref{cor:asdim_product} and \cref{prop:asdim_quotient}, $\asdim \big(F\backslash \mcY_F\big)\leq |F|\cdot \big(|F|\asdim (\mcY) + 1 \big)-1$. Thus, the family $F\backslash \mcY_F$ has finite asymptotic dimension. Therefore, by definition, to prove that $F\backslash\mcX$ is in $\fR_\alpha$ it suffices to show that for every bounded subfamily $\mcB$ of $F\backslash \mcY_F$, the inverse image under $\overline{G}_{F}$ is in $\fR_\beta$, for some $\beta < \alpha$. 
Using the commutative diagram above, every $A \in (P \circ G_{F})^{-1}(\mcB)$ is an $F$-invariant subspace of some $X\in\mcX$. Thus, $F\backslash(P \circ G_{F})^{-1}(\mcB)=Q\big((P \circ G_{F})^{-1}(\mcB)\big)=(\overline{G}_{F})^{-1}(\mcB)$. Therefore, if we show that $(P \circ G_{F})^{-1}(\mcB)$ lies in $\fR_\beta$ for some $\beta<\alpha$, then it will follow from the induction assumption that $(\overline{G}_{F})^{-1}(\mcB)$ is in $\fR_\beta$, as desired. 

Let $R$ be the uniform bound on the diameters of the elements of $\mcB$. For each $B \in \mcB$, pick a point $z\in p^{-1}(B)$, then
\[p^{-1}(B)\subseteq \bigcup_{h\in F} B_{2R}(hz)\subseteq \prod_{F}Y.\]
For $k\in F$, let $\pi_{k}\colon \prod_FY\to Y$ be projection to the $k$th coordinate. Let $B_{h,k}:=B_{2R}(\pi_{k}(hz))\subseteq Y$. Then
\[B_{2R}(hz)\subseteq \prod_{k\in F}B_{h,k}.\]
Suppose that $g_{F}(x)\in B_{2R}(hz)$. Then, by definition, 
\[ d^1\big(g_{F}(x),hz\big)=\sum_{k\in F} d_{Y}\big(g(k^{-1}x),\pi_k(hz)\big)\leq 2R, \]
and so $d_{Y}\big(g(x),\pi_e(hz)\big)\leq 2R$, where $e$ denotes the identity element of $F$. That is, $g(x) \in B_{h,e}$.
Therefore, $g_F^{-1}(B_{2R}(hz))\subseteq g^{-1}(B_{h,e})$ for each $h\in F$. Thus, 
\[ (p\circ g_{F})^{-1}(B) \subseteq g_{F}^{-1}\left(\bigcup_{h\in F}B_{2R}(hz) \right)= \bigcup_{h\in F}g_{F}^{-1}(B_{2R}(hz)) \subseteq \bigcup_{h\in F}g^{-1}(B_{h,e}).\]
This implies that there exist bounded subfamilies $\mcB_1,\dots,\mcB_{|F|}$ of $\mcY$ such that every metric space in $(P \circ G_{F})^{-1}(\mcB)$ is a subspace of $\bigcup_{j=1}^{|F|} X_j$, for some $X_j \in G^{-1}(\mcB_j)$, $1\leq j \leq |F|$.  By the definition of $G$, since each $\mcB_j$ is bounded there is a $\beta<\alpha$ such that $G^{-1}(\mcB_j) \in \fR_\beta$ for every $j$, $1\leq j \leq |F|$. Since $\fR_\beta$ satisfies Finite Union Permanence (\cref{thm:finunion1}) and Coarse Permanence (\cref{thm:coarseinv}), $(P \circ G_{F})^{-1}(\mcB)$ is also in $\fR_\beta$. This completes the proof. 
\end{proof}

We immediately obtain the following corollary by applying Finite Quotient Permanence for $\fR$ (using the version from \cref{prop:FQPfamily}) to the family $\{G_F~|~ F\leq G, |F|< \infty\}$, where $G_F=G$ for every $F$, and $F$ acts on $G_F$ by left translation.

\begin{cor}\label{thm:quotients}
Let $G$ be a countable group that has regular FDC and a global upper bound on the orders of its finite subgroups. Then the metric family $\{F\backslash G~|~F\leq G, |F|< \infty \}$ has regular FDC.
\end{cor}

In order to apply the first author's injectivity results for algebraic $K$- and $L$-theory \cite[Theorems 8.1 and 9.1]{KasFDC} to a countable group $G$ that has FDC, a finite dimensional model for $\underbar EG$ (the universal space for proper $G$-actions) and a global upper bound on the orders of its finite subgroups, one must verify that the metric families $\{F\backslash G~|~F\leq G, |F|< n \}$ have FDC for every $n\in\bbN$. The first author achieved this in \cite[Theorem 4.13]{KasFDC} for every finitely generated subgroup of $GL_n(R)$, where $R$ is a commutative ring with unit. The proof depended on a very technical proof that a stronger version of this is true for solvable groups. Furthermore, that proof cannot be generalized to elementary amenable groups.
It is possible to circumvent this technical condition for solvable groups and prove injectivity for linear groups by using the powerful theorem that solvable groups satisfy the Farrell--Jones Conjecture. (For details see \cite[Theorem 1.1]{KasLin}, where injectivity is proved for every countable subgroup $G$ of a linear group with a finite dimensional model for $\underbar EG$.) 
However, since all of these classes of groups have regular FDC (\cref{thm:group_examples}) and regular FDC implies FDC (\cref{lem:FDC}), \cref{thm:quotients} yields a unified proof that the metric families $\{F\backslash G~|~F\leq G, |F|< n \}$ have FDC for such groups for every $n\in\bbN$. In this sense, \cref{thm:quotients} shows that regular FDC is a useful concept for establishing split injectivity of assembly maps and might lead to new injectivity results once FDC is known for more classes of groups (see \cref{thm:inj}).


\bibliographystyle{amsalpha}
\bibliography{FDC}


\end{document}